\documentclass[custombib, a4paper,12pt,times,print,index]{PhDThesisPSnPDF}

\ifsetCustomMargin
  \RequirePackage[left=37mm,right=30mm,top=35mm,bottom=30mm]{geometry}
  \setFancyHdr 
\fi

\ifsetCustomFont
  \RequirePackage{helvet}
\fi

\RequirePackage[labelsep=space,tableposition=top]{caption}

\usepackage{subcaption}

\usepackage{booktabs} 
\usepackage{multirow}

\usepackage{amsfonts}
\usepackage{amsmath}
\usepackage{amssymb}
\usepackage{siunitx} 
\usepackage{amsthm}
\usepackage{comment}
\usepackage{bm}
\usepackage{bbm}

\ifuseCustomBib

\RequirePackage[backend=bibtex, style=ieee, citestyle=numeric-verb, sorting=nty, natbib=true]{biblatex}
\nocite{*}
\fi

\newtheorem{definition}{Definition}[section]
\newtheorem{proposition}[definition]{Proposition}

\newtheorem{corollary}[definition]{Corollary}
\newtheorem{theorem}[definition]{Theorem}
\newtheorem{lemma}[definition]{Lemma}
\newtheorem{remark}[definition]{Remark}
\newtheorem{notation}[definition]{Notation}

\usepackage[T1]{fontenc}
\usepackage[english]{babel}
\usepackage[latin1]{inputenx}
\usepackage[pdftex]{graphicx}           
\usepackage{setspace}                   
\usepackage{indentfirst}                
\usepackage{makeidx}                    
\usepackage[nottoc]{tocbibind}          
\usepackage{courier}                    
\usepackage{type1cm}                    
\usepackage{listings}                   
\usepackage{titletoc}
\usepackage{enumerate}
\usepackage[all]{xy}
\usepackage{quotchap}
\usepackage{url}
\usepackage{hyphenat}
\usepackage{fancyhdr}
\usepackage{xcolor}
\usepackage{bbold}
\usepackage[utf8]{inputenc}
\usepackage{csquotes}
\usepackage{ stmaryrd }

\usepackage{lmodern}
\usepackage{titlesec}
\usepackage{microtype}
\usepackage{tikz}

\definecolor{myblue}{RGB}{0,0,0} 

\titleformat{\chapter}[display]
  {\normalfont\bfseries\color{myblue}}
  {\filleft%
    \begin{tikzpicture}
    \node[
      outer sep=0pt,
      text width=2.5cm,
      minimum height=3cm,
      fill=myblue,
      font=\color{white}\fontsize{80}{90}\selectfont,
      align=center
      ] (num) {\thechapter};
    \node[
      rotate=90,
      anchor=south,
      font=\color{black}\Large\normalfont
      ] at ([xshift=-5pt]num.west) {\textls[180]{\textsc{\chaptertitlename}}};  
    \end{tikzpicture}%
  }
  {10pt}
  {\titlerule[2.5pt]\vskip3pt\titlerule\vskip4pt\Large\centering\sffamily}

\newcommand{\calgebra}{\mathfrak{A}}
\newcommand{\nalgebra}{\mathfrak{M}}
\newcommand{\algebra}{\mathfrak{S}}
\newcommand{\nanalytic}{\mathfrak{M}_\mathcal{A}}

\newcommand{\hilbert}{\mathcal{H}}
\newcommand{\iu}{i\mkern1mu}
\newcommand{\Dom}[1]{\mathcal{D}\left(#1\right)}
\newcommand{\ie}{\textit{i}.\textit{e}. }
\newcommand{\eg}{\textit{e}.\textit{g}.}
\newcommand{\strip}[1]{\mathcal{D}_{#1}}
\newcommand{\ip}[2]{\left\langle #1 , #2 \right \rangle}

\newcommand{\cchull}[1]{\overline{conv}\left(#1\right)}

\makeatletter
\newcommand*{\defeq}{\mathrel{\rlap{%
			\raisebox{0.3ex}{$\m@th\cdot$}}%
		\raisebox{-0.3ex}{$\m@th\cdot$}}%
	=}
\makeatother

\DeclareMathOperator{\Ker}{Ker}
\DeclareMathOperator{\Ran}{Ran}

\let\Re\relax
\newcommand{\Re}[1]{\operatorname{Re}\left(#1\right)}
\let\Im\relax
\newcommand{\Im}[1]{\operatorname{Im}\left(#1\right)}

\title{Lecture Notes on Noncommutative $L_p$ Spaces}

\author{Ricardo Correa da Silva}

\dept{Department of Physics}

\university{University of S\~ao Paulo}
\crest{\includegraphics[width=0.3\textwidth]{brasao_usp_cor}}

\college{University of São Paulo}

\subject{Operator Algebras} 

\ifdefineAbstract
\fi

\begin{document}

\frontmatter



\begin{titlepage}
    \begin{center}
        \vspace*{1cm}
        
        \fontsize{22}{0}
        \textbf{Lecture Notes on Non-Commutative $L_p$ Spaces}

        \vspace{0.5cm} 
        
        \vspace{2.5cm}
        
        \fontsize{18}{0}
        \textbf{Ricardo Correa da Silva}
                
        \vspace{2 cm}
        
		
        \vspace{3cm}
        
        
        
		\vfill
		
		S\~ao Paulo, 2016
        
    \end{center}
\end{titlepage}




\tableofcontents




\printnomencl

\mainmatter
\renewcommand\thechapter{P}
\chapter{Preamble}

These notes started been written during my studies on the subject and have been largely improved to be used as lecture notes in a mini-course on "Noncommutative $L_p$-spaces: the tracial case" at the Basque Center for Applied Mathematics (BCAM). This text has no intention of being a complete treatment on the topic, it is just a brief presentation which contains some known results of special interest of the author.

Here we intend to present a detailed construction of noncommutative measure and a first approach to noncomutative $L_p$-spaces. By the end, we intend to give an idea of the general constructions by U. Haagerup, and H. Araki and T. Masuda.

The majority of the results presented here can be found in M. Terp's lecture notes, \cite{terp81}, or in Q. Xu's lecture notes, \cite{Xu07}, with some minor modifications to suit the author's style or to give more details in the proofs.

We expect the reader is familiar with Functional Analysis and has some knowledge of concepts such as von Neumann Algebras, GNS-Contruction and Measure Theory. There are very good books on these topics such as \cite{Blackadar2006}, \cite{Bratelli1}, \cite{Bratelli2}, \cite{KR83}, and \cite{KR86}.

This test is divided in two chapters: the first one is devoted to present noncommutative measure, and the second one to noncommutative $L_p$-spaces, starting with the tracial case, in a quite detailed way, and finishing with just a general idea of the non-tracial one.

\renewcommand\thechapter{I}
\chapter{Introduction}


In Takesaki words, ``noncommutative measure and integration refers to the theory of weights, states, and traces on von Neumann algebras". It is quite clear that the explanation for this name relies on the Riesz-Markov-Kakutani Theorem, which puts Measure Theory in the light of Functional Analysis.

For any functional analyst who is told about the existence of such a thing called ``noncommutative integration'', a natural question arises: how about ``noncommutative $L_p$-spaces?''.
 
The first attempt to answer this question was done by E. Segal in \cite{Segal53}. In his work, Segal defined $L_1$ and $L_2$ in von Neumann algebras which admit a trace.
Later, J. Dixmier, in \cite{Dixmier53}, extended the concept for $1<p\leq \infty$, again for tracial von Neumann algebra.

Many years later, U. Haagerup was able to give a definition of noncommutative $L_p$-spaces including the type III algebras. After some works of A. Connes and Hilsum, a equivalent definition of noncommutative $L_p$-spaces was proposed by H. Araki and T. Masuda in \cite{Araki82}\footfullcite{Araki82} using the Hilbert space and based in Tomita-Takesaki (Relative) Modular Theory, but unfortunately not much seems to be known about this approach.

These topics have been object of extensively investigation, specially noncommutative measure which is the basis to define other objects of interest in Physics and Mathematics, \eg Noncommutative Geometry, and Noncommutative Probability.
\renewcommand{\thechapter}{\arabic{chapter}}
\setcounter{chapter}{0}

\chapter{Measurability with respect to a trace}  

\ifpdf
    \graphicspath{{Chapter1/Figs/Raster/}{Chapter1/Figs/PDF/}{Chapter1/Figs/}}
\else
    \graphicspath{{Chapter1/Figs/Vector/}{Chapter1/Figs/}}
\fi

In this chapter we will construct the noncommutative $L_p$-space in the particular case in which the von Neumann algebra has a trace.

\section{Positive Operators and Functionals}

To start this section we need the definition of positive operator and functionals.

\begin{definition}
	Let $\calgebra$ be a $C^\ast$-algebra. $A\in\calgebra$ is said to be positive if
	$$\left\|\mathbbm{1}-\frac{A}{\|A\|}\right\|\leq1.$$
	We denote by $\calgebra_+$ the set of positive operators in the algebra.
\end{definition}
\begin{definition}[Positive Linear Functional and States]
	\
	\begin{enumerate}[(i)]
		
		\item Let $\calgebra$ be a $C^\ast$-algebra and $\omega$ a linear functional on $\calgebra$, $\omega$ is said to be positive \index{functional !positive} if
		$$\omega(A)\geq 0 \ \forall A\in \calgebra_+.$$
		
		\item Let $\calgebra$ be a $C^\ast$-algebra and state a positive linear functional on $\calgebra$, $\omega$ is said to be a state \index{state} if
		$$\|\omega\|=1.$$
	\end{enumerate}
\end{definition}

\begin{remark}
	If $\calgebra$ is unital we also have $\omega(\mathbbm{1})=1$, because if $\omega$ is a state on $\calgebra$, $\|\omega\|=1 \Leftrightarrow \omega(\mathbbm{1}) =1$. Moreover, no continuity is required in the definition because it becomes a consequence of positiveness.
	
\end{remark}

\begin{proposition}[Cauchy-Schwarz inequality]
	\label{cauchyschwarz}
	Let $\calgebra$ be a $C^\ast$-algebra and $\omega$ a state. Then
	$$\begin{aligned}
	\omega(A^\ast B)	&=\overline{\omega(B^\ast A)}\\
	\left|\omega(A^\ast B)\right|^2 &\leq \omega(A^\ast A)\omega(B^\ast B)
	\end{aligned}$$
\end{proposition}
\begin{proof}
	It is a consequence of positivity that, for all $\lambda\in\mathbb{C}$,
	$$0\leq \omega\left((A+\lambda B)^\ast(A+\lambda B)\right)=\omega(A^\ast A)+\overline{\lambda}\omega(B^\ast A)+\lambda\omega(A^\ast B)+|\lambda|^2\omega(B^\ast B).$$
	
	In Particular, for $\lambda=1,\iu$ we conclude, respectively, that $\Im{\omega(A^\ast B)}=-\Im{\omega(B^\ast A)}$ and $\Re{\omega(A^\ast B)}=\Re{\omega(B^\ast A)}$. Hence $\omega(A^\ast B)=\overline{\omega(B^\ast A)}$ and
	$$0\leq \omega(A^\ast A)+\overline{\lambda\omega(A^\ast B)}+\lambda\omega(A^\ast B)+|\lambda|^2\omega(B^\ast B).$$
	
	If $\omega(B^\ast B)=0$ the inequality is trivial. Assuming $\omega(B^\ast B)\neq0$ and taking \linebreak $\lambda=-\frac{\overline{\omega(A^\ast B)}}{\omega(B^\ast B)}$ we have that
	$$0\leq \omega(A^\ast A)-\frac{|\omega(A^\ast B)|^2}{\omega(B^\ast B)}.$$
	
\end{proof}

\begin{proposition}
	\label{PPC}
	Let $\omega$ be a linear functional on a unital $C^\ast$-algebra $\calgebra$. The following are equivalent
	\begin{enumerate}[(i)]
		\item $\omega$ is positive.
		\item $\omega$ is continuous and $\|\omega\|=\omega(\mathbbm{1})$.
	\end{enumerate} 
\end{proposition}
\begin{proof}
	$(i)\Rightarrow(ii)$ First, let us prove that $\omega(A)$, with $A$ positive and norm-one operator, are uniformly bounded by some constant $M>0$. In fact, if this is not true, there exists a sequence of positive elements $A_n$, $\|A_n\|=1$, such that $\omega(A_n)>n 2^n$. But then $\displaystyle B_m=\sum_{n=1}^{m} 2^{-n} A_n$ is a norm convergent sequence to some positive element $B$ and it follows from positivity that $$n=\omega(B_m)\leq \omega(B) \ \forall n\in \mathbb{N}.$$
	
	It is a contradiction and we can define
	$$M=\sup\{\omega(A) | A\geq0,\ \|A\|=1\}.$$
	
	It follow from the polarization identity that for each $\displaystyle A\in \calgebra$ with $\|A\|=1$, there exists $A_i\in \calgebra_+$ with $\|A_i\|\leq 1$ and $\displaystyle A=\sum_{j=0}^{3}\iu^jA_j$. Thus, $|\omega(A)|\leq 4M$ which means $\omega$ is continuous.
	
	Furthermore, $\|\omega\|\geq \omega(\mathbbm{1})$. On the other hand, by the Cauchy-Schwarz inequality
	$$|\omega(A)|^2= |\omega(A\mathbbm{1})|^2\leq |\omega(A)| \|\omega(\mathbbm{1})| \Rightarrow |\omega(A)|\leq |\omega(\mathbbm{1})|.$$
	
	$(ii)\Rightarrow(i)$ Dividing by its norm, we can suppose $\|\omega\|=1$, and by hypotheses, $\omega(\mathbbm{1})=1$.
	
	Now, let $A$ be a positive element, then
	$$\left\|\mathbbm{1}-\frac{A}{\|A\|}\right\|\leq 1.$$
	
\end{proof}


\section{Weights}

A weight is, in some way, a natural generalization of positive linear functionals where the co-domain is the positive extended real line $\overline{\mathbb{R}}$, as will be clear soon.

\begin{definition}
	A weight\index{weight} on a $C^*$ -algebra $\calgebra$ is a function $\phi:\calgebra_+ \to \overline{\mathbb{R}}_+$ such that
	\begin{enumerate}[(i)]
		\item $\phi(\lambda A)=\lambda \phi(A) \quad \forall A\in \calgebra_+$ and $\forall \lambda \geq 0$;
		\item $\phi(A+B)=\phi(A)+\phi(B)  \quad \forall A, B\in \calgebra_+$
	\end{enumerate}
\end{definition}

It is important to stress that we use the convention $\infty\cdot 0=0$. 

\begin{definition}
	\label{weightsets}
	Let $\phi$ be a weight on a $C^\ast$-algebra $\calgebra$, we define:
	\begin{enumerate}[(i)]
		\item $ \displaystyle \mathfrak{N}_\phi=\mathfrak{D}^2_\phi=\{A \in \calgebra \ | \ \phi(A^\ast A)<\infty\}$
		\item $ \displaystyle \mathfrak{F}_\phi=\{A \in \calgebra_+ \ | \ \phi(A)<\infty\}$
		\item $ \displaystyle \mathfrak{M}_\phi=\mathfrak{D}^1_\phi=span\left[\mathfrak{F}_\phi \right]$
		\item $ \displaystyle N_\phi= \{A \in \calgebra \ | \ \phi(A^\ast A)=0\}$
	\end{enumerate}
\end{definition}

Some important definitions and results for a thorough understanding of the following, such as the Krein-Milman Theorem, can be found in the standard literature.

\begin{proposition}
	\label{simplepropweights}
	The following properties hold:
	\begin{enumerate}[(i)]
		\item $ \displaystyle \mathfrak{F}_\phi \subset \mathfrak{N}_\phi\cap \mathfrak{N}_\phi^\ast $.
		\item $ \displaystyle \mathfrak{F}_\phi$ is a face of $\calgebra_+$.
		\item  $ \displaystyle \mathfrak{N}_\phi$ and $ \displaystyle N_\phi$ are left ideals of $\calgebra$.
		\item $ \displaystyle \mathfrak{M}_\phi=\mathfrak{N}_\phi^\ast \mathfrak{N}_\phi=span\left[\{A^\ast A | A \in \mathfrak{N}_\phi\}\right]$.
		\item $\displaystyle \mathfrak{M}_\phi$ is a sub-$\ast$-algebra of $\calgebra$ and $\mathfrak{M}_+=\mathfrak{M}\cap \calgebra_+=\mathfrak{F}_\phi$.
	\end{enumerate}
\end{proposition}
\begin{proof}
	$(i)$ It is obvious that $\mathfrak{F}_\phi \subset \mathfrak{N}_\phi$, since for each $A\in \mathfrak{F}_\phi$, $A^\ast A=AA\leq \|A\|A$ hence
	$$0\leq \phi(\|A\|A-A^\ast A)=\|A\|\phi(A)-\phi(A^\ast A) \Rightarrow \phi(A^\ast A)\leq \|A\|\phi(A)<\infty.$$
	Using now the positiveness we also get $\mathfrak{F}_\phi \subset \mathfrak{N}_\phi^\ast$
	
	$(ii)$ Trivial.
	
	$(iii)$ Notice that $B^\ast A^\ast A B\leq \|A\|^2 B^\ast B$ because, if we call $D$ the unique positive square root of $\|A\|^2-A^\ast A$, we have
	$$\|A\|^2 B^\ast B-B^\ast A^\ast A B=B^\ast\left(\|A\|^2-A^\ast A\right)B=(DB)^\ast (DB)\geq0.$$
	
	It follows that, for $A, B\in \calgebra$,
	$$\phi\left((AB)^\ast AB\right)=\phi(B^\ast A^\ast A B)\leq \|A\|^2 \phi(B^\ast B).$$
	
	Moreover, it is easy to see that $\lambda \in \mathbb{C}, A\in \mathfrak{N}_\phi \Rightarrow \lambda A \in \mathfrak{N}_\phi$, furthermore 
	$$(A+B)^\ast (A+B)+(A-B)^\ast (A-B)=2A^\ast A +2 B^\ast B.$$
	hence if we take $A,B \in \mathfrak{N}_\phi$ we must have
	$$\infty>\phi\left(2A^\ast A+2B^\ast B\right)=\phi\left((A+B)^\ast (A+B)+(A-B)^\ast (A-B)\right)\geq\phi\left((A+B)^\ast (A+B)\right),$$
	from which $A+B \in \mathfrak{N}_\phi$
	
	The analogous properties for $N_\phi$ are trivial.
	
	$(iv)$ The polarization identity states that
	$$B^\ast A=\frac{1}{4} \sum_{n=0}^{3}\iu^n(A+i^n B)^\ast(A+\iu^n B),$$
	but the proof of $(iii)$ above says that $ A+i^n B \in \mathfrak{N}_\phi$, $0\leq n\leq 3$, and the conclusion holds.
	
	$(v)$ From the definition, every $A\in \mathfrak{M}$ can be written as a linear combination $A=A_1-A_2+\iu A_3-\iu A_4$ with $A_i\in \mathfrak{F}_\phi$. If $A\geq 0$, then $0\leq A=A_1-A_2\leq A_1 \Rightarrow A\in \mathfrak{F}_\phi$.
	
\end{proof}

\begin{remark}
	At that point it is important to note that a weight $\phi$ admits a natural linear extension $\tilde{\phi}$ to $\mathfrak{M}_\phi$. It is a simple consequence of $(iv)$ in Proposition \ref{simplepropweights} and the uniqueness of polarization identity.
	
	There will be no distinction between $\phi$ and $\tilde{\phi}$ in the following.
\end{remark}

Notice that these properties make the quotient $\mathfrak{N}_\phi/N_\phi$ a pre-Hilbert space, provided with $\ip{A}{B}_\phi=\phi(B^\ast A)$. We call the completion of this space $\hilbert_\phi$.

\begin{definition}
	A weight $\phi$ on a $C^*$-algebra $\calgebra$ is said to be:
	\begin{enumerate}[(i)]
		\item densely defined\index{weight! densely defined}  if $\mathfrak{F}_\phi$ is dense in $\calgebra_+$;
		\item faithful\index{weight! faithful} if $\phi(A^\ast A)=0 \Rightarrow A=0$;
		\item normal\index{weight !normal} if $\phi(\sup A_i)=\sup \phi(A_i)$ for all bounded increasing net $(A_i)_i \in \calgebra_+$.
		\item semifinite\index{weight! semifinite} if $\mathfrak{M}_\phi$ is weakly dense in $\calgebra$;
		\item a trace \index{weight! trace} if $\phi(A^\ast A)=\phi(A A^\ast)$.

	\end{enumerate}
\end{definition}

\section{Traces and Continuity} 

\begin{lemma}
	Let $\nalgebra$ be a von Neumann algebra, $\mathfrak{I}\subset \nalgebra$ a two-sided ideal and $A\in\left(\overline{\mathfrak{I}}^{WOT}\right)^+$. Then there exists a increasing net $(A_i)_{i\in I} \subset\mathfrak{I}^+$ converging to $A$.
\end{lemma}
\begin{proof}
	First, to simplify the notation, lets write $\mathfrak{I}^+ \setminus\{0\}=\{A_j\}_{j\in J}$. Now, consider the set
	$$\mathcal{F}=\left\{K\subset J \ \middle| \ \sum_{k\in K_f}A_k \leq A, \ \forall K_f \subset K, K_f \mbox{ finite}\right\}$$
	provided with the partial order $K_1\preceq K_2 \Leftrightarrow K_1\subset K_2$.
	It is easy to see that it satisfies the requirements to conclude by Zorn's Lemma the existence of a maximal element $I\subset F$.
	
	Consider $\displaystyle B=A-\sup_{i \in I} A_i\geq0$ and the $(p_j)_{j\in J} \subset \mathfrak{I}$ a increasing net of projections which converges to the identity of $\overline{\mathfrak{I}}^{WOT}$. By the ideal property, $B^\frac{1}{2}p_jB^\frac{1}{2} \in \mathfrak{I}^+ \ \forall j \in J$ but it does not correspond to any index in $I$, in fact, if we had $B^\frac{1}{2}p_jB^\frac{1}{2}=A_r$ we will also had the inequality $$\sum_{i\in I_f}A_i \leq B+(B-A)\leq A, \ \forall I_f \subset K\cup\{r\}, I_f \mbox{ finite,}$$
	but the maximality of $I$ forbids it. The only remaining possibility is $$B^\frac{1}{2}p_jB^\frac{1}{2} =0, \quad \forall j \in J \Rightarrow B=0.$$
\end{proof}

\begin{remark}
	3Notice that by Vigier's theorem the net $(A_i)_{i\in I}$ described in the previous result converges to $A$ in the SOT.
\end{remark}

\begin{lemma}
	\label{positiveneighbourhood}
	Let $\nalgebra$ be a von Neumann algebra, $A \in \nalgebra_+$ and $\phi,\psi$ two normal positive functionals such that $\phi(A)<\psi(A)$. Then, there exists $B\in \nalgebra_+\setminus\{0\}$ such that $$\phi(C)<\psi(C), \quad \forall C\leq B, C\in\nalgebra^+\setminus\{0\}.$$
\end{lemma}
\begin{proof}
	Let
	$$\mathcal{F}=\left\{P\in\nalgebra^+ \ \middle| \ P\leq A, \ \phi(P)\geq\psi(P)\right\}.$$
	
	Consider a chain $\{P_i\}_{i\in I}\subset \mathcal{F}$, we know $\displaystyle P=\sup_{i\in I}P_i\leq A$ and, by normality,
	$$\phi(P)=\sup_{i\in I} \phi(P_i)\geq \sup_{i\in I} \psi(P_i)=\psi(P).$$
	
	Since every chain has a maximal element, the Zorn's lemma says $\mathcal{F}$ has a maximal element $Q$.
	
	Take $B=A-Q\in \nalgebra_+$. Of course $B\leq A$ and, since $$\psi(B)=\psi(A)-\psi(Q)>\phi(A)-\phi(Q)=\phi(B)\geq0,$$
	$B\neq0$. Finally, for every $P\in\nalgebra^+$, $P\leq B$, we must have $\phi(P)<\psi(P)$, otherwise the positive element $A-P\geq Q$ would violate the maximality of $Q$.
	
\end{proof}

\begin{lemma}
	\label{ultraweakcontinuous}
	Let $\nalgebra$ be a von Neumann algebra and $\omega$ a linear functional on $\nalgebra$. Then $\omega$ is ultra-weakly (weakly) continuous if, and only if, it is ultra-strongly (strongly) continuous. Furthermore, for every ultra-strongly functional there exists $(x_n)_n, (y_n)_n \in \hilbert$ satisfying $\displaystyle\sum_{n \in\mathbb{N}}\|x_n\|^2, \sum_{n \in\mathbb{N}}\|y_n\|^2 < \infty$ such that
	$$\omega(A)=\sum_{n \in\mathbb{N}}\ip{y_n}{Ax_n}, \quad \forall A \in \nalgebra.$$
	
\end{lemma}
\begin{proof}
	It is obvious that ultra-weakly continuous functionals are ultra-strongly continuous.
	
	For the other implication, by taking a basic neighbourhood in the ultra-strongly topology, there exists a sequence $(x_n)_n \in \hilbert$ with $\displaystyle\sum_{n\in \mathbb{N}}\|x_n\|^2 < \infty$ such that \sloppy ${\displaystyle |\omega(A)|^2<\sum_{n\in \mathbb{N}}\|Ax_n\|^2}$.
	
	We can define now define $\displaystyle\tilde{\hilbert}=\bigoplus_{n\in \mathbb{N}}\hilbert$ and the norm-continuous functional
	$$\tilde{\omega}\left((Ax_n)_n\right)=\omega(A)$$
	defined on the linear subspace $\left\{ (Ax_n)_n \in \tilde{\hilbert} \ | \ A \in \nalgebra\right\}$. This functional can be extended to the whole space $\tilde{\hilbert}$ (still denoted $\tilde{\omega}$) using the Hahn-Banach's Theorem and the Riesz's Theorem ensures the existence of a vector $(y_n)_n \in \tilde{\hilbert}$ (with means $\displaystyle \sum_{n\in \mathbb{N}}\|y_n\|^2 < \infty$) such that
	$$\tilde{\omega}\left((Ax_n)_n\right)=\omega(A)=\ip{(y_n)_n}{(Ax_n)_n}_{\tilde{\hilbert}}=\sum_{n \in\mathbb{N}}\ip{y_n}{Ax_n}.$$
	
	This establishes not only the ultra-weakly continuity of $\omega$, but it gives us the mentioned general form of the ultra-weakly (and ultra-strongly, since they are the same) continuous functionals.
	
	The case when $\omega$ is strongly-continuous is basically the same.
	
\end{proof}

\begin{corollary}
	\label{WOT=SOT}
	Let $K\in \nalgebra$ be a convex set. Then, the ultra-weak (weak) closure and the ultra-strong (strong) closure of $K$ coincides.
\end{corollary}
\begin{proof}
	Of course $\overline{K}^{SOT}\subset \overline{K}^{WOT}$, so the thesis is equivalent to show $\overline{K}^{WOT}\setminus \overline{K}^{SOT}=\emptyset$.
	
	Suppose there exists $x\in\overline{K}^{WOT}\setminus \overline{K}^{SOT}$. Then $\overline{K}^{SOT}-x$ is a closed convex subset that does not contain the null-vector. Let $U$ be a convex balanced neighbourhood of zero such that $\left(x+\overline{U}^{SOT}\right)\cap \overline{K}^{SOT}=\emptyset$. Therefore $\overline{K}^{SOT}+\overline{U}^{SOT}-x$ is also a closed convex subset that does not contain the null-vector, by Corollary \ref{CTMD}, there exists a SOT-continuous functional $x^\ast$ such that $x^\ast(x)=0$ and $x^\ast(k)>0, \ \forall k\in \overline{K}^{SOT}+\overline{U}^{SOT}$. 
	
	Then, there exists $\alpha>0$ such that
	$x^\ast(x)=0$ and $x^\ast(k)>\alpha>0, \ \forall k\in \overline{K}^{SOT}$, but it is not possible, once the previous lemma stated that such a $x^\ast$ is also WOT-continuous and $x\in \overline{K}^{WOT}$.
	
\end{proof}

\begin{proposition}
	Let $\nalgebra$ be a von Neumann algebra, $\phi$ a positive functional on $\nalgebra$. The following conditions are equivalent:
	\begin{enumerate}[(i)]
		\item $\phi$ is normal;
		\item $\phi$ is ultra-strongly continuous;
		\item there exists $(x_n)_n\subset \hilbert$, such that $\displaystyle\sum_{n=1}^{\infty}\|x_n\|^2 <\infty$ and $\displaystyle \phi=\sum_{n=1}^{\infty}\omega_{x_n}$.
	\end{enumerate}
\end{proposition}
\begin{proof}
	Let
	$$\mathcal{F}=\left\{P\in\nalgebra^+ \ \middle| \ P\leq1 \mbox{ and } A \mapsto \phi(AP) \mbox{ is ultra-weakly continuous}\right\}.$$
	
	Consider a chain $\{P_i\}_{i\in I}\subset \mathcal{F}$, we know $\displaystyle P=\sup_{i\in I}P_i \in \nalgebra^+$, $P_i \xrightarrow{SOT} P$ by the Vigier's theorem, so it also converges in the $WOT$, but on the unitary ball the weak and ultra-weak topologies coincide. Furthermore, for $\|A\|\leq 1$,
	$$\left|\phi\left(A(P-P_i)\right)\right|^2\leq \phi(A^\ast A)\phi(P-P_i)\leq \phi(1)\phi(P-P_i)$$
	
	Thus $A\mapsto \phi(AP)$ is the uniform limit of the ultra-weakly continuous positive functionals $\left(A \mapsto \phi(AP_i)\right)_{i\in I}$ on the unitary ball, thus it is also ultra-weakly continuous on the unitary ball. Finally, linearity it is ultra-weakly continuous on $\nalgebra$.
	
	Applying now the Zorn's lemma, there exists a maximal positive operator $Q\in \mathcal{F}$, $Q\leq 1$, such that $A\mapsto \phi(AQ)$ is ultra-weakly continuous.
	
	Remains to prove $Q=1$, in fact if $1-Q>0$ we can take an ultra-weak positive functional $\psi$, such that $0\leq \phi(1-Q)<\psi(1-Q)$ by choosing a $z \in \hilbert$ such that $\phi(1-Q)<\ip{(1-Q)z}{z}$ and $\psi=\left(A\mapsto \ip{Az}{z}\right)$), for example.
	
	By Lemma \ref{positiveneighbourhood} there exists $B\in\nalgebra^+\setminus\{0\}$ such that $B\leq 1-Q$ and
	$$\phi(P)<\psi(P) \ \forall C\leq B, \  C\in\nalgebra^+\setminus\{0\}.$$
	
	Hence, since for each $A\in\nalgebra\setminus\{0\}$ we have $BA^\ast A B \leq \|A\|^2B B \leq \|A\|^2 \|B\| B$, the Cauchy-Schwarz inequality gives us
	$$\left|\phi\left(\frac{AB}{\|A\|\|B\|^\frac{1}{2}}\right)\right|^2\leq\phi(1)\phi\left(\frac{BA^\ast A B}{\|A\|^2\|B\|}\right)<\psi\left(\frac{BA^\ast A B}{\|A\|^2\|B\|}\right)\Leftrightarrow \left|\phi\left(AB\right)\right)|^2<\psi\left(BA^\ast A B\right).$$
	
	Notice know that if $(A_i)_{i\in I}\subset\nalgebra$ is a net such that $A_i\to0$ ultra-strongly, then, for every $x,y\in \hilbert$, $\ip{B A_i^\ast A_i B x}{y}= \ip{A_i B x}{A_i B y}\leq\|A_i Bx\| \|A_i B y\| \to 0$. Thus, $\phi$ is ultra-strongly continuous, and by Lemma \ref{ultraweakcontinuous} $\phi$ is ultra-weakly continuous, since ultra-strongly and ultra-weakly functionals coincides on $\nalgebra$ , so $A\mapsto\phi(AB)$ is ultra-weakly continuous.
	
	This contradicts the maximality of $Q$, since $A\mapsto \phi\left(A(Q+B)\right)$ is ultra-strongly continuous and $Q+B\leq Q+1-Q=1$ and the conclusion holds.
	
\end{proof}

\section{Measure with respect to a trace}

Henceforth we will use $\tau$ and call it simply a trace, but we mean a normal, faithful, semifinite trace. It is important to note that supposing the existence of such a trace we restrict our options of algebras to the semifinite ones.

\begin{definition}
	Let $\nalgebra\subset B(\hilbert)$ be a von Neumann algebra, we say a linear operator $A:\Dom{A} \to \hilbert$ is affiliated to $\nalgebra$ if, for every unitary operator $U\in\nalgebra^\prime$, $UAU^\ast=A$. We denote that an operator is affiliated by $A\eta\nalgebra$ and the set of all affiliated operator by $\nalgebra_\eta$.
\end{definition}

A very interesting way to look at this definition is through spectral projections. By this definition, the spectral projections of an affiliated operator belong to the von Neumann algebra. In fact, an equivalent condition to an operator $A$ being affiliated to a von Neumann algebra is that the partial isometry in the polar decomposition of $A$ and all the spectral projections of $|A|$ lie in the von Neumann algebra. This guarantees that an affiliated operator can be approached by algebra elements in the spectral sense.

\begin{lemma}
	Let $\nalgebra$ be a von Neumann algebra and $A\eta \nalgebra$
	\begin{enumerate}[(i)]
		\item If $A$ is closable, then $\overline{A}\eta\nalgebra$.
		\item If $A$ is densely defined, then $A^\ast\eta \nalgebra$.
	\end{enumerate}
\end{lemma}
\begin{proof}
	$(i)$ The condition just warranty the existence of $\overline{A}$. Since it exists $U\in\nalgebra^\prime$, $UAU^\ast=A \Rightarrow U\overline{A}U^\ast=\overline{A}$.
	
	$(ii)$ The condition plays the same role and, since $A^\ast$ exists, $U\in\nalgebra^\prime$, $UAU^\ast=A \Rightarrow UA^\ast U^\ast=A^\ast$.
	
\end{proof}

\begin{definition}
	Let $\varepsilon, \delta>0$, define $$D(\varepsilon,\delta)=\left\{A\in \nalgebra_\eta \ \middle| \ \exists p\in \nalgebra_p \textrm{ such that } p\hilbert\subset\Dom{A}, \ \|Ap\|\leq\varepsilon \textrm{ and } \tau(1-p)\leq\delta\right\}$$
\end{definition}

Some important properties we are about to present rely on equivalence of projections. One of these important equivalences is $s_R^\nalgebra(A)\sim s_R^\nalgebra(A)$, it is, the right and left support of an operator are equivalent projections. This standard result is a consequence of polar decomposition, since for the partial isometry, $u$, in the polar decomposition of $A$ holds $uu^\ast= s_L^\nalgebra(A)$ and $u^\ast u= s_R^\nalgebra(A)$. We can used it to deduce another important equivalence as it follow
$$s^\nalgebra_L(q(1-p)=proj[\Ran\left(q(1-p)\right]=q-(p\wedge q)$$
$$proj[\Ker\left(q(1-p)\right]=p+(1-p)\wedge (1-q)\Rightarrow s^\nalgebra_R(q(1-p)=1-p-(1-p\vee q)= (p\vee q)-p.$$

Then $(p\vee q)-p \sim q-(p\wedge q)$. It also has a interesting consequence: if two projections $p$ and $q$ are such that $(p\wedge q)=0$ (which means the intersection of the ranges is the null space), then $p\lesssim 1-q$.

\begin{proposition}
	\label{Dsumandproduct}
	Let $\varepsilon_1,\varepsilon_2, \delta_1, \delta_2>0$, then
	\begin{enumerate}[(i)]
		\item $D(\varepsilon_1,\delta_1)+D(\varepsilon_2,\delta_2)\subset D(\varepsilon_1+\varepsilon_2,\delta_1+\delta_2)$;
		\item $D(\varepsilon_1,\delta_1)D(\varepsilon_2,\delta_2)\subset D(\varepsilon_1\varepsilon_2,\delta_1+\delta_2)$.
	\end{enumerate}
\end{proposition}
\begin{proof}
	$(i)$ Let $A_i\in D(\varepsilon_i,\delta_i)$, $i=1,2$. By definition, there exist projections $p_i$ such that $p_i\hilbert\subset\Dom{A_i}$,$\|Ap_i\|\leq\varepsilon_i$ and $\tau(1-p_i)\leq\delta_i$. Taking $q=p_1\wedge p_2$ we get $q\hilbert \subset\Dom{A_1}\cap\Dom{A_2}=\Dom{A_1+A_2}$, $\|(A_1+A_2)q\|\leq \|A_1 q\|+\|A_2 q\|\leq \|A_1 p_1\|+\|A_1 p_2\|\leq\varepsilon_1+\varepsilon_2$ and $$\begin{aligned}
	\tau(1-q) &= \tau(1-p_1\wedge p_2)\\
			&=\tau\left((1-p_1)\vee(1 -p_2)\right)\\
			&\leq \tau(1-p_1)+\tau(1-p_2)\leq \delta_1+\delta_2.
	\end{aligned}$$

$(ii)$ Let $A_i$, $p_i$, $i=1,2$, as before. Note that $(1-p_1)A_2)p_2\in \nalgebra$, then we can define $r=[\Ker{\left((1-p_1)A_2 p_2\right)}]$.

It follow from this definition that, for every $x \in r\hilbert$, $A_2p_2 x\in p_1\hilbert\subset\Dom{A_1}$. This means $r\hilbert\subset\Dom{A_1A_2p_2}$ which would imply $(p_2\wedge r)\hilbert\subset \Dom{A_1A_2}$ and 
$$\begin{aligned}
\|A_1A_2(p_2\wedge r)\|	&=\|A_1p_1A_2p_2(p_2\wedge r)\|\\
&\leq \|A_1p_1\|\|A_2p_2\|\\
&=\varepsilon_1\varepsilon_2.
\end{aligned}$$

Now, $1-r=s^\nalgebra_R\left((1-p_1)A_2p_2\right) \sim s^\nalgebra_L\left((1-p_1)A_2p_2\right)\leq 1-p_1$, and thus
$$\begin{aligned}
\tau(1-p_2\wedge r)&=\tau\left((1-p_2)\vee(1-r)\right)\\
&\leq \tau(1-p_2)+\tau(1-p_1)\\
&\leq\delta_1+\delta_2.
\end{aligned}$$

\end{proof}

\begin{proposition}
	\label{proposition d<->tau}
	Let $A\eta\nalgebra$ a closed densely defined operator, $\varepsilon, \delta>0$ and \\ $E_{(\varepsilon,\infty)}=s^\nalgebra_L\left((|A|-\varepsilon)_+\right)$ the spectral projection of $|A|$ on the interval ${(\varepsilon,\infty)}$. Then
	$$A\in D(\varepsilon,\delta) \Leftrightarrow \tau\left(E_{(\varepsilon,\infty)}\right)\leq\delta$$
\end{proposition}
\begin{proof}
$(\Leftarrow)$ Just take $p=1-E_{(\varepsilon,\infty)}$, then $p\hilbert \subset\Dom{|A|}$, $\|Ap\|=\|A(1-E_{(\varepsilon,\infty)})\|\leq \varepsilon$ and by hypothesis $\tau(1-p)\leq\delta$.

$(\Rightarrow)$ Suppose $A\in D(\varepsilon,\delta)$, there exists $p\in\nalgebra_p$ such that $p\hilbert\subset\Dom{A}$, $\|Ap\|\leq \varepsilon$ and $\tau(1-p)\leq\delta$. Let us take $\left\{E_{(\lambda,\infty)}\right\}_{\lambda \in \mathbb{R}_+}$ the spectral projections of $|A|$. Notice that $\||A|E_{(\varepsilon,\infty)}x\|>\varepsilon\|E_{(\varepsilon,\infty)}x\| \ \forall x \in \hilbert$, while $\||A|p x\|\leq \varepsilon \|px\|$. It follows that $\left(E_{(\varepsilon,\infty)} \wedge p\right)=0$ or, in other words, $E_{(\varepsilon,\infty)}\lesssim 1-p$.

Finally, the previous inequality leads us to $\tau(E_{(\varepsilon,\infty)})\leq\tau(1-p)\leq \delta$.

\end{proof}

\begin{proposition}
	\label{AinD=>A*inD}
	Let $A\eta \nalgebra$ be a closed densely defined operator.
	$$A\in D(\varepsilon,\delta)\Leftrightarrow A^\ast \in D(\varepsilon,\delta).$$
\end{proposition}
\begin{proof}
	Let $A=u|A|$  be the polar decomposition of $A$, so $|A^\ast|^2=A A^\ast=u|A|\left(u|A|\right)^\ast=u|A|^2u^\ast$.
	
	Using the previous identity, where $E_{(\varepsilon,\infty)}^{|A|}$ and $E_{(\varepsilon,\infty)}^{|A^\ast|}$  are the spectral projections of $|A|$ and $|A^\ast|$ respectively, we have $$\begin{aligned}
			\tau\left(E_{(\varepsilon,\infty)}^{|A^\ast|}\right)&=\tau\left(u E_{(\varepsilon,\infty)}^{|A^\ast|} u^\ast\right)\\
			&=\tau\left(E_{(\varepsilon,\infty)}^{|A|}\right)\\
			&\leq\delta.
	\end{aligned}$$ 
\end{proof}

\begin{definition}
A subspace $V\subset\hilbert$ is $\tau$-dense if $\forall \delta>0$, there exists $p\in\nalgebra_p$ such that $p\hilbert\subset V$ and $\tau(1-p)<\delta$.	
\end{definition}

\begin{proposition}
	\label{pn->1}
A subspace $V\subset\hilbert$ is $\tau$-dense if, and only if, there exists an increasing sequence of projections $(p_n)_n\subset\nalgebra_p$ such that $p_n\to 1$ and $\tau(1-p_n)\to 0$ and $\displaystyle \bigcup_{n\in\mathbb{N}} p_n\hilbert\subset V$.
\end{proposition}

\begin{proof}
	Of course the existence of such a sequence of projections implies the $\tau$-density of $V$.
	
	On the other hand, if $V$ is $\tau$-dense, for each $n\in\mathbb{N}$, there exists a projection $q_n$ such that $q_n\hilbert\subset V$ and ${\tau(1-q_n)<2^{-n}}$.
	
	Now, in order to obtain an increasing sequence, define
	$\displaystyle p_n=\bigwedge_{k\geq n}q_k$.

	To show that $p_n\to 1$, define $p=\displaystyle \bigvee_{n\in\mathbb{N}}q_n$, then, for all $n\in \mathbb{N}$,
	\begin{equation}\begin{aligned}
	\label{increasingtau}
	\tau(1-p)&\leq \tau(1-p_n)\\ &=\tau\left(1-\bigwedge_{k\geq n}q_k\right)\\
	&=\tau\left(\bigvee_{k\geq n}(1-q_k)\right)\\
	&\leq \sum_{k\geq n}\tau(1-q_k)\\
	&\leq \sum_{k\geq n}2^{-k} \\
	&=2^{1-n}
	\end{aligned}\end{equation}
but this is only possible if $\tau(1-p)=0$ and since $\tau$ is faithful, $p=1$. It follows by the Vigier's Theorem that $p_n\to1$ in the SOT.
 
\end{proof}

\begin{corollary}
If $V\subset \hilbert$ is a $\tau$-dense subspace, $V$ is dense in $\hilbert$.
\end{corollary}

\begin{corollary}
Let $V_1,V_2 \subset\hilbert$ be $\tau$-dense subspaces. Then $V_1\cap V_2$ is $\tau$-dense. 	
\end{corollary}
\begin{proof}
First, although we are going to present a proof based on the previous proposition, it is quite interesting to notice that this result has already been proved. In fact, the proof can be found in Proposition \ref{Dsumandproduct}.

Let $(p^i_n)_n\in \nalgebra_p$ such that $tau(1-p^i_n)\to0$ and $\displaystyle \bigcup_{n\in\mathbb{N}} p^i_n\hilbert\subset V_i$. Define $q_n=p^1_n\wedge p^2_n$. Then
$$\begin{aligned}
\tau(1-q_n) &= \tau(1-p^1_n\wedge p^2_n)\\
&=\tau\left((1-p^1_n)\vee(1 -p^2_n)\right)\\
&\leq \tau(1-p^1_n)+\tau(1-p^2_n) \to 0.\\
\end{aligned}$$

Furthermore, $q=p^1_n\wedge p^2_n$ is the projection in the intersection of the image of $p_1$ and $p_2$, which is a subset of $V_1\cap V_2$, for every $n\in\mathbb{N}$.

\end{proof}

\begin{definition}[Balanced Weight]
	Let $\nalgebra$ be a von Neumann algebra and $\phi, \psi$ weights. We define the \index{weight! balanced} balanced weight
	$$\begin{aligned}
	\theta_{\phi,\psi}:	& \quad M_{2\times2}(\nalgebra)	&\to	& \ \overline{R}\\
	& \begin{pmatrix} A_{11} & A_{12} \\ A_{21} & A_{22}\end{pmatrix} &\mapsto	& \ \phi(A_{11})+\psi(A_{22}).
	\end{aligned}$$
\end{definition}

\begin{notation}
	To simplify the notation we will not use $\theta_{\phi,\psi}$, instead, we will only use $\theta$ when it is clear what $\phi$ and $\psi$ are.
	
	In addition, it is much easier to write $\displaystyle \sum_{i,j=1}^{2}A_{ij}e_{ij}$ to substitute $\begin{pmatrix} A_{11} & A_{12} \\ A_{21} & A_{22}\end{pmatrix}$, where
	$$e_{11}=\begin{pmatrix} \mathbb{1} & 0 \\ 0 & 0\end{pmatrix}, \quad
	e_{12}=\begin{pmatrix} 0 & \mathbb{1} \\ 0 & 0\end{pmatrix}, \quad
	e_{21}=\begin{pmatrix} 0 & 0 \\ \mathbb{1} & 0\end{pmatrix}, \quad
	e_{22}=\begin{pmatrix} 0 & 0 \\ 0 & \mathbb{1}\end{pmatrix}. \quad$$
\end{notation}

\begin{lemma}
	Let $\nalgebra$ be a von Neumann algebra and $\phi, \psi$ weights.
	\label{lemmabalanced}
	\begin{enumerate}[(i)]
		\item $\theta_{\phi,\psi}$ is a normal semifinite weight if, and only if, $\phi, \psi$ have these properties;
		\item $\theta_{\phi,\psi}$ is faithful if, and only if, $\phi, \psi$ are faithful;
	\end{enumerate}
\end{lemma}

\begin{proposition}
	Let $\nalgebra\subset B(\hilbert)$ be a von Neumann algebra and let $A_1, A_2\in \nalgebra_\eta$ be closed (densely defined) operators such that there exists a $\tau$-dense subspace $V\subset \Dom{A_1}\cap\Dom{A_2}$ where $\left. A_1\right|_V=\left. A_2\right|_V$. Then $A_1=A_2$.
	
\end{proposition}

\begin{proof}
	
Note that $\theta=\theta_{\tau,\tau}$ is a trace in $M_{2\times2}(\nalgebra)$.

Let $p_i$ be the projection on the graph of $A_i$, $i=1,2$. Notice that $$M_{2\times2}(\nalgebra)^\prime=\left\{\begin{pmatrix} A^\prime & 0 \\ 0 & A^\prime\end{pmatrix}\middle| A^\prime \in \nalgebra^\prime\right\} \textrm{ and }$$
$$
\begin{pmatrix} A^\prime & 0 \\ 0 & A^\prime\end{pmatrix}\begin{pmatrix} x  \\  A_i x\end{pmatrix}=
\begin{pmatrix} A^\prime x  \\  A^\prime A_i x\end{pmatrix}=
\begin{pmatrix} A^\prime x  \\  A_i A^\prime x\end{pmatrix}\in \Gamma(A_i)
$$
hence, $p_i \in M_{2\times2}(\nalgebra)^{\prime\prime}=M_{2\times2}(\nalgebra)$.

By hypothesis, there exists $p\in \nalgebra_p$ such that $p\hilbert\subset V$ and $\tau(1-p)<\delta$.
	
Take $r=s^\nalgebra_L(p_1-p_2)=[\Ran(p_1-p_2)]$ and notice $r\wedge p=0$. Thus $r\leq 1-p\oplus p$ and it follows that $\theta(r\oplus r)\leq\theta\left((\mathbb{1}-p)\oplus(1-p)\right)=\tau(\mathbb{1}-p)+\tau(\mathbb{1}-p)<2\delta$. Since $\delta>0$ is arbitrary, $\theta(r\oplus r)=0\Rightarrow p_1=p_2$.

\end{proof}

\begin{definition}
	Let $\nalgebra$ be a von Neumann algebra and $\tau$ be a trace, a closed (densely defined) operator $A\in \nalgebra_\eta$ is said $\tau$-measurable if $\Dom{A}$ is $\tau$-dense. We denote by $\nalgebra_\tau$ the set of all $\tau$-measurable operators.
\end{definition}

Notice that by the previous proposition, if $A$ is a $\tau$-measurable operator and $B$ extends $A$, we must have $A=B$. This, in turn, implies that a $\tau$-measurable symmetric operator is self-adjoint.

\begin{definition}
	\label{defpremeasurable}
	Let $\nalgebra$ be a von Neumann algebra and $\tau$ be a trace. An operator $A\eta\nalgebra$ is said $\tau$-premeasurable if, $\forall \delta>0$, there exists $p\in\nalgebra_p$ such that $p\hilbert \subset \Dom{A}$, $\|Ap\|<\infty$ and $\tau(1-p)\leq \delta$.  
\end{definition}

An equivalent way to define a $\tau$-premeasurable operator relies on $D(\varepsilon,\delta)$: $A$ is $\tau$-premeasurable if, and only if, $\forall \delta>0$, there exists $\varepsilon>0$ such that $A\in D(\varepsilon,\delta)$.

Another interesting thing to notice is that a $\tau$-premeasurable operator is densely defined since $D(A)$ must be $\tau$-dense.

\begin{proposition}
\label{equivalencetaumeasurability}
Let $A \eta \nalgebra$ be a closed densely defined operator and $\left\{E_{(\lambda,\infty)}\right\}_{\lambda\in\mathbb{R}_+}$ the spectral decomposition of $|A|$. The following are equivalent:
\begin{enumerate}[(i)]
	\item $A$ is $\tau$-measurable;
	\item $|A|$ is $\tau$-measurable;
	\item $\forall \delta>0 \ \exists \varepsilon>0$ such that $A\in D(\varepsilon,\delta)$;
	\item $\forall \delta>0 \ \exists \varepsilon>0$ such that $\tau\left(E_{(\varepsilon,\infty)}\right)<\delta$;
	\item $\displaystyle \lim_{\lambda\to\infty}\tau\left(E_{(\lambda,\infty)}\right)=0$;
	\item $\exists \lambda_0>0$ such that $\tau\left(E_{(\lambda_0,\infty)}\right)<\infty$.
\end{enumerate}
\end{proposition}
\begin{proof}

$(i)\Leftrightarrow(iii)$ Simply rewrite the definition, as mentioned before;

$(i)\Leftrightarrow(ii)$ Its is just to notice that $A\in D(\varepsilon,\delta)\Leftrightarrow |A|\in D(\varepsilon,\delta)$, which follow from Proposition \ref{proposition d<->tau};

$(iii)\Leftrightarrow(iv)$ It is Proposition \ref{proposition d<->tau};

$(iv)\Leftrightarrow(v)$ $\left(E_{(\lambda,\infty)}\right)_{\lambda\in\mathbb{R}}$ is a decreasing net of projections. Let $\delta>0$, there exists $\varepsilon>0$ such that $\tau\left(E_{(\varepsilon,\infty)}\right)<\delta$. Hence, for every $\lambda>\varepsilon$, $\tau\left(E_{(\varepsilon,\infty)}\right)<\tau\left(E_{(\lambda,\infty)}\right)<\delta$ and the other implication is analogous;

$(v)\Rightarrow(vi)$ is obvious;

$(vi)\Rightarrow (v)$ Let $\lambda_0>0$ such that $\tau\left(E_{(\lambda_0,\infty)}\right)<\infty$. Define the increasing upper bounded net $\left(E_{(\lambda_0,\infty)}-E_{(\lambda,\infty)}\right)_{\lambda>\lambda_0}$, notice $\left(\tau\left(E_{(\lambda_0,\infty)}-E_{(\lambda,\infty)}\right)\right)_{\lambda>\lambda_0}\subset \mathbb{R}_+$ is also an increasing net and $\left(\tau\left(E_{(\lambda,\infty)}\right)\right)_{\lambda>\lambda_0}\subset \mathbb{R}_+$ is a decreasing net, hence, both nets have limits. By normality of $\tau$

$$\begin{aligned}
\lim_{\lambda\to \infty}\tau\left(E_{(\lambda_0,\infty)}-E_{(\lambda,\infty)}\right)&=\sup_{\lambda>\lambda_0}\tau\left(E_{(\lambda_0,\infty)}-E_{(\lambda,\infty)}\right)\\
&=\tau\left(E_{(\lambda_0,\infty)}-\bigwedge_{\lambda >\lambda_0}E_{(\lambda,\infty)}\right)\\
&=\tau\left(E_{(\lambda_0,\infty)}\right).\\
\end{aligned}$$

\end{proof}



\chapter{Noncommutative $L_p$-spaces}

Hitherto, we have presented the theory of noncommutative measure which enable us to start presenting now the first approach to noncommutative spaces.

\section{Segal-Dixmier's noncommutative $L_p$-spaces}
\begin{proposition}
	$\nalgebra_\tau$ provided with the usual scalar operations and involution, and the following vector operations is a $\ast$-algebra.
	\begin{enumerate}[(i)]
		\item $A\bm{+}B=\overline{A+B}$
		\item $A\bm{\times}B=\overline{AB}$
	\end{enumerate}
\end{proposition}

\begin{proof}
	First of all, in order that the previous definition makes sense, we must guarantee that for every $A,B\in \nalgebra_\tau$, $A+B$ and $AB$ are closable. In fact for $A,B\in \nalgebra_\tau$ and for every $\delta>0$ the previous proposition guarantees that there exist $\varepsilon_A, \varepsilon_B>0$ such that
	$$\begin{aligned}
	A\in D\left(\varepsilon_A,\frac{\delta}{2}\right) &\Rightarrow A^\ast\in D\left(\varepsilon_A,\frac{\delta}{2}\right)\\
	B\in D\left(\varepsilon_B,\frac{\delta}{2}\right) &\Rightarrow B^\ast\in D\left(\varepsilon_B,\frac{\delta}{2}\right),
	\end{aligned}$$
	where the implication is a consequence of Proposition \ref{AinD=>A*inD}. Also, it mean $A^\ast \in \nalgebra_\tau$. Then, by Proposition \ref{Dsumandproduct},
	$$\begin{aligned}
	A^\ast+B^\ast &\in D\left(\varepsilon_A+\varepsilon_B,\delta\right)\\
	A^\ast B^\ast &\in D\left(\varepsilon_A\varepsilon_B,\delta\right)\\
	\end{aligned}$$
	
	These inclusions, as commented after Definition \ref{defpremeasurable}, means $A^\ast+B^\ast$ and $B^\ast A^\ast$ are $\tau$-premeasurables, which implies they are densely defined. Hence $A+B\subset (A^\ast+B^\ast)^\ast$ and $AB\subset (B^\ast A^\ast)^\ast$ admit closed extensions, so they are closable.
	
	Now, $\overline{A+B}$ and $\overline{AB}$ are closed densely defined operators, for which condition $(iii)$ in the Proposition \ref{equivalencetaumeasurability} holds thanks to Proposition \ref{Dsumandproduct}, hence $A\bm{+}B, A\bm{\times}B \in \nalgebra_\tau$. 
	
	Just remains to prove the identities for $\bm{+}, \bm{\times}$ and their relations with $*$. Let $A,B,C \in \nalgebra_\tau$, we know each of operators in the following equalities are closed and they coincide on a $\tau$-dense subspace, then the equality holds:
	$$\begin{aligned}
	(A\bm{+}B)\bm{+}C&=A\bm{+}(B\bm{+}C)	&\qquad (A\bm{\times}B)\bm{\times}C&=A\bm{\times}(B\bm{\times}C)\\
	(A\bm{+}B)\bm{\times}C&=A\bm{\times}C\bm{+}B\bm{\times}C &\qquad C\bm{\times}(A\bm{+}B)&=C\bm{\times}A\bm{+}C\bm{\times}B\\
	(A\bm{+}B)^\ast &= A^\ast\bm{+}B^\ast &\qquad (A\bm{\times}B)^\ast&=B^\ast\bm{\times}A^\ast\\
	\end{aligned}$$
	
\end{proof}

From now on, we will differentiate the symbols $\bm{+}, \bm{\times}$ and the usual sum and multiplication of operators only if it may cause a misunderstanding.

\begin{proposition}
	$\nalgebra_\tau$ is a complete Hausdorff topological $\ast$-algebra with respect to the topology generated by the system of neighbourhoods of zero $\left\{\nalgebra_\tau \cap D(\varepsilon,\delta)\right\}_{\varepsilon>0, \delta>0}$. Furthermore, $\nalgebra$ is dense in $\nalgebra_\tau$. We will denote the balanced absorbing neighbourhood of zero $N(\varepsilon,\delta)=\nalgebra_\tau \cap D(\varepsilon,\delta)$.
\end{proposition}

\begin{proof}
	It is easy to verify $N(\varepsilon,\delta)$ is in fact balanced and absorbing, furthermore, Proposition \ref{Dsumandproduct} implies that, for every $\varepsilon_1,\varepsilon_2>0$ and $\delta_1, \delta_2>0$, $N\left(\frac{\varepsilon_1}{2},\frac{\delta_1}{2}\right)+N\left(\frac{\varepsilon_1}{2},\frac{\delta_1}{2}\right)=N\left(\varepsilon_1,\delta_1\right)$ and $N\left(\min\{\varepsilon_1,\varepsilon_2\},\min\{\delta_1,\delta_2\}\right)\subset N\left(\varepsilon_1,\delta_1\right)\cap N\left(\varepsilon_2,\delta_2\right)$. Hence, there exists a unique vector topology such that $\left\{N(\varepsilon,\delta)\right\}_{\varepsilon>0, \delta>0}$ is a basis of neighbourhoods of zero.
	
	In order to show this topology is Hausdorff, let $A,B \in \nalgebra_\tau$ be two distinct operators. For each $\delta>0$, define $$\varepsilon_\delta=\inf\left\{\epsilon\in\mathbb{R}_+\middle| A-B \in N(\epsilon,\delta)\right\},$$
	notice that there exists $\tilde{\delta}>0$ such that $\varepsilon_{\tilde{\delta}}>0$, because if we had $\varepsilon_\delta=0$ for every $\delta>0$ would exist projections $p_\delta^n$ for each $n\in \mathbb{N}$ such that $\|(A-B)p_\delta^n\|<\frac{1}{n}$ and $\tau(1-p_\delta^n)\leq \delta$. Defining $p_\delta=\inf_{n\in\mathbb{N}} p_\delta^n$ we would have $\|(A-B)p_\delta\|=0$ and $\tau(1-p_\delta)\leq \delta$, which implies $A$ and $B$ coincide on a $\tau$-dense subspace, but this is not possible since $A\neq B$. It is easy to see $B\notin A+N\left(\frac{\varepsilon_{\tilde{\delta}}}{2},\tilde{\delta}\right)$.
	
	Let us prove the two neighbourhoods $A+N\left(\frac{\varepsilon_{\tilde{\delta}}}{4},\frac{\tilde{\delta}}{2}\right)$ and $B+N\left(\frac{\varepsilon_{\tilde{\delta}}}{4},\frac{\tilde{\delta}}{2}\right)$ of $A$ and $B$ respectively are disjoint. In fact, suppose it is not, there exist $T_1, T_2 \in N\left(\frac{\varepsilon_{\tilde{\delta}}}{4},\frac{\tilde{\delta}}{2}\right)$ such that $A+T_1=B+T_2 \Rightarrow B=A+T_1-T_2 \in A+N\left(\frac{\varepsilon_{\tilde{\delta}}}{2},\tilde{\delta}\right)$, but it is not possible.
	
	Of course the vector space operations are continuous, the involution is continuous due to Proposition \ref{AinD=>A*inD}, and for the product consider $A,B\in\nalgebra_\tau$ and $AB+N(\varepsilon,\delta)$ a basic neighbourhood of $AB$. There exists $\alpha>\varepsilon$ such that $A \in N\left(\alpha,\frac{\delta}{6}\right)$ and $B\in N\left(\alpha,\frac{\delta}{6}\right)$, then, if $\tilde{A}\in A+N\left(\frac{\varepsilon}{3\alpha},\frac{\delta}{6}\right)$ and $\tilde{B}\in B+N\left(\frac{\varepsilon}{3\alpha},\frac{\delta}{6}\right)$, we have
	$$\begin{aligned}
	AB-\tilde{A}\tilde{B}&=-(A-\tilde{A})(B-\tilde{B})+A(B-\tilde{B})+(A-\tilde{A})B\\
	&\in N\left(\frac{\varepsilon^2}{9\alpha},\frac{\delta}{3}\right)+
	N\left(\frac{\varepsilon}{3},\frac{\delta}{3}\right)+
	N\left(\frac{\varepsilon}{3},\frac{\delta}{3}\right)
	\subset N(\varepsilon,\delta).
	\end{aligned}$$
	
	For the density, we use Proposition \ref{pn->1}. Let $A\in \nalgebra_\tau$ and $(p_n)_n\subset \nalgebra_p$ be an increasing sequence of projections such that $p_n\to 1$, $\tau(1-p_n)\to 0$ and $\displaystyle \bigcup_{n \in \mathbb{N}}p_n\hilbert\subset \Dom{A}$. Thus $(Ap_n)_n\in\nalgebra$ and $Ap_n\to A$ since the product is continuous and $p_n\xrightarrow[]{\tau} 1$.
	
	It just remains to show $\nalgebra_\tau$ is complete. Notice the space has a countable basis of neighbourhoods since $\left\{N\left(\frac{1}{n},\frac{1}{m}\right)\right\}_{n\in\mathbb{N}^\ast,m\in\mathbb{N}^\ast}$ is such a countable basis. This means we just have to proof every Cauchy sequence is convergent.
	
	Let $(A_n)_n \in \nalgebra_\tau$ a Cauchy sequence. Since $\overline{\nalgebra}^\tau=\nalgebra_\tau$ there exists $(A^\prime_n)_n \subset \nalgebra$ such that $A_n-A^\prime_n\in N\left(\frac{1}{n},\frac{1}{n}\right)$, hence $(A^\prime_n)_n$ is also a Cauchy sequence.
	
	For each $k\in \mathbb{N}$ there exists $N_k\in \mathbb{N}$ such that, $\forall n,m\geq N_k$, there exists a projection $q_k$ with $\|(A^\prime_m-A^\prime_n)q_k\|\leq 2^{-k}$ and $\tau(1-q_k)\leq 2^{-k}$.
	
	Define $p=\displaystyle \bigvee_{n\in\mathbb{N}}q_n$. Simply repeating calculation (\ref{increasingtau}) we see that $(p_n)_n\subset\nalgebra_p$, defined by $\displaystyle p_n=\bigwedge_{k\geq n}q_k$, is an increasing sequence of projections such that $\tau(1-p_k)\leq 2^{1-k}$.
	
	Let's now define $\displaystyle \Dom{A}=\bigcup_{k\in\mathbb{N}}p_k\hilbert$. Notice that $x\in \Dom{A}$ implies the existence of $j\in\mathbb{N}$ for which $x\in p_j\hilbert\subset p_l\hilbert$ for all $l\geq j$. Since for any $k>j$ and $n,m \geq N_k$ we have $$\begin{aligned}
	\|(A^\prime_n-A^\prime_m)x\|&= \|(A^\prime_n-A^\prime_m)p_k x\|\\
	&\leq\|(A^\prime_n-A^\prime_m)p_k\| \| x\|\\
	&\leq \|(A^\prime_n-A^\prime_m)q_k\| \| x\|\\
	& \leq 2^{-k}\|x\|,
	\end{aligned}$$
	hence $(A^\prime_nx)_{n\geq j} \subset \hilbert$ is a Cauchy sequence. Then, we can define $A:\Dom{A}\to\hilbert$ by 
	$$ Ax=\lim_{j<n\to \infty} A^\prime_n x.$$
	
	Since $\Dom{A}$ is $\tau$-dense and $A\eta \nalgebra$ by construction, it remains to show it is closable and $A_n\xrightarrow{\tau} A$. In order to prove it is closable, notice $(A^{\prime \ast}_n)_n\in\nalgebra$ is again a Cauchy sequence, so we can repeat the previous construction to obtain $B:\Dom{B}\to\hilbert \in\nalgebra_\tau$ defined by $Bx=\lim A^{\prime \ast}_n x$. Hence, $\forall x \in \Dom{A}$ and $\forall y \in \Dom{B}$,
	$$\ip{Ax}{y}=\lim_{n\to\infty}\ip{A^\prime_nx}{y}	=\lim_{n\to \infty}\ip{x}{A_n^{\prime \ast}y}	=\ip{x}{By},$$
	which means $A\subset B^\ast$, hence $A$ is closable and $\overline{A}\in \nalgebra_\tau$ by Proposition \ref{equivalencetaumeasurability}. 
	
	Finally, for the convergence notice that for every $\varepsilon, \delta>0$ there exists $k\in \mathbb{N}$ such that $2^{-k}<\varepsilon$ and $2^{1-k}<\delta$, so, $\tau(1-p_k)<\delta$ and, $\forall n\geq N_k$,
	$$\begin{aligned}
	\|(\overline{A}-A^\prime_n)p_k\|&=\sup_{\|x\|\leq1}\left\|(A-A^\prime_n)p_k x\right\|\\
	&=\sup _{\|x\|\leq1}\left\|\lim_{m\to\infty} A^\prime_m p_k x -A^\prime_n p_k x\right\|\\
	&=\sup _{\|x\|\leq1}\lim_{m\to\infty}\left\| A^\prime_m p_k x -A^\prime_n p_k x\right\|\\
	&\leq\sup _{\|x\|\leq1}\lim_{m\to\infty}\left\| (A^\prime_m -A^\prime_n) p_k \right\|\|x\|\\
	&\leq\lim_{m\to\infty}\left\| (A^\prime_m -A^\prime_n) p_k \right\|\\
	&\leq 2^{-k}\\
	&<\varepsilon.\\
	\end{aligned}$$
	
\end{proof}

\begin{lemma}
	\label{normtraceinequality}
	Let $\nalgebra$ be a von Neumann algebra, $\tau$ trace on $\nalgebra$, $A\in \nalgebra$ and $B\in \mathfrak{M}_\tau$, then
	$$\left|\tau(AB)\right|\leq\tau(|AB|)\leq \|A\|\tau(|B|) $$
\end{lemma}
\begin{proof}
	First, lets prove the lemma in the case $A,B\in \mathfrak{M}_\tau^+$.
	
	Let $D=\sqrt{\|A\|\mathbb{1}-A}$ then
	$$0\leq \left(DB^\frac{1}{2}\right)^\ast \left(DB^\frac{1}{2}\right)=B^\frac{1}{2}\left(\|A\|\mathbb{1}-A\right)B^\frac{1}{2},$$
	from the trace positiveness it follows that
	$$0\leq \tau\left(B^\frac{1}{2}\left(\|A\|\mathbb{1}-A\right)B^\frac{1}{2}\right)=\|A\|\tau(B)-\tau\left(B^\frac{1}{2}A B^\frac{1}{2}\right)=\|A\|\tau(B)-\tau\left(A B\right).$$
	
	For the general statement, let $A=u|A|$ and $B=v|B|$ the polar decomposition of $A$ and $B$. Using the Cauchy-Schwartz's inequality we obtain
	$$	\begin{aligned}
	\left|\tau(AB)\right|^2&=\left|\tau(u|A|v|B|)\right|^2\\
	&=\left|\tau\left((|B|^\frac{1}{2}u|A|^\frac{1}{2})(|A|^\frac{1}{2}v|B|^\frac{1}{2})\right)\right|^2\\
	&\leq \tau\left(\left(|B|^\frac{1}{2}u|A|^\frac{1}{2}\right)^\ast \left(|B|^\frac{1}{2}u|A|^\frac{1}{2}\right)\right) \tau\left(\left(|A|^\frac{1}{2}v|B|^\frac{1}{2}\right)^\ast \left(|A|^\frac{1}{2}v|B|^\frac{1}{2}\right)\right)\\
	&= \tau\left(|A|^\frac{1}{2}u^\ast |B|u|A|^\frac{1}{2})\right) \tau\left(|B|^\frac{1}{2}v^\ast |A|v|B|^\frac{1}{2}\right)\\
	&= \tau\left(|B|u|A|u^\ast\right) \tau\left(|A|v|B|v^\ast\right)\\
	&= \tau\left(|B||A^\ast|\right) \tau\left(|A||B^\ast|\right)\\
	\end{aligned}$$
	Now, from the previous result for positive operators we conclude that, $\forall A\in\nalgebra, B \in \nalgebra_\tau$,
	$$\left|\tau(AB)\right|^2\leq\tau\left(|B||A^\ast|\right) \tau\left(|A||B^\ast|\right)\leq\|A^\ast\|\tau(|B|)\|A\|\tau(|B^\ast|)=\|A\|^2\tau(|B|)^2.$$
	
\end{proof}

Let us start to prove important inequalities and then define the non-commutative $L_p$-spaces.

\begin{theorem}[H\"older Inequality]
	\label{holder}
	Let $\nalgebra$ be a von Neumann algebra and $\tau$ a normal faithful semifinite trace in $\nalgebra$, let also $A,B\in\nalgebra$ and $p,q>1$ such that $\frac{1}{p}+\frac{1}{q}=1$, then
	$$\tau(|AB|)\leq \tau(|A|^p)^\frac{1}{p}\tau(|B|^q)^\frac{1}{q}.$$
\end{theorem}
\begin{proof}
	First note that if $\tau\left(|A|^p\right)=0$,  $\tau\left(|A|^p\right)=\infty$, $\tau\left(|B|^q\right)=0$, or $\tau\left(|B|^q\right)=\infty$ the inequality is trivial.
	
	On the other hand, if $0<\tau\left(|A|^p\right),\tau\left(|B|^q\right)<\infty$ we are able to define, for every $n\in\mathbb{N}$, $$|A|_n^p=\frac{1}{\tau\left(|A|^p\right)}\int^{\|A\|}_\frac{1}{n}\lambda^p dE^{|A|}_\lambda \textrm{ and } |B|_n^q=\frac{1}{\tau\left(|B|^q\right)}\int^{\|B\|}_\frac{1}{n} \lambda^q dE^{|B|}_\lambda$$
	where $\left\{E^{|A|}_\lambda\right\}_{\lambda\in \mathbb{R}^+}$ and $\left\{E^{|B|}_\lambda\right\}_{\lambda\in \mathbb{R}^+}$ are the spectral resolutions of $|A|$ and $|B|$, respectively. These definition guarantees $|A|_n^p \to \frac{|A|^p}{\tau\left(|A|^p\right)^{\frac{1}{p}}}$ and $|B|_n^q \to \frac{|B|^q}{\tau\left(|B|^q\right)^{\frac{1}{q}}}$ monotonically, $\tau(|A|^p_n), \tau(|B|^q_n)\leq 1$.

	Let $A=u|A|$, $B=v|B|$ and $AB=w|AB|$ be the respective polar decompositions of these operators.
	Since, for every $\varepsilon>0$, $\sigma\left(\varepsilon\mathbb{1}+|A|\right),\sigma\left(\varepsilon\mathbb{1}+|B|\right)\subset [\varepsilon,\infty)$, we can define a function $f_n:\mathbb{C} \to \mathbb{C}$ by
	$$\begin{aligned}
	f_n(z)&=\tau\left(w^\ast u\left(|A|_n^p\right)^z v \left(|B|_n^q\right)^{(1-z)}\right)\\
	&=\tau\left(\left(|B|_n^q\right)^{-z} w^\ast u\left(|A|_n^p\right)^z v |B|_n^q\right)
	\end{aligned}$$
	
	To see that this function is entire analytic, notice that there exists $A_n \in \nalgebra$ such that $$\left(|B|_n^q\right)^{-z}|B|_n w^\ast u\left(|A|_n^p\right)^z v = \sum_{j=1}^{\infty} C_j z^j$$
	which can be obtained simply as a Taylor series. It follows from the Lemma \ref{normtraceinequality} that, for any $\varepsilon>0$, there exists $N\in \mathbb{N}$ large enough such that 
	$$\begin{aligned}
	\left|f_n(z)-\sum_{j=1}^{N} \tau(C_j |B|_n^q ) z^j\right|&=\left|\tau\left(\left(\left(|B|_n^q\right)^{-z} |B|_n w^\ast u\left(|A|_n^p\right)^z v-\sum_{j=1}^{N} C_j z^j\right) |B|_n^q\right)\right|\\
	&\leq\left\|\left(|B|_n^q\right)^{-z} |B|_n w^\ast u\left(|A|_n^p\right)^z v-\sum_{j=1}^{N} C_j z^j\right\|\tau\left( |B|_n^q\right)\\
	&<\varepsilon\\
	\end{aligned}.$$
	
	This function is also bounded in the strip $\{z\in \mathbb{C} \ | \ 0\leq \Re(z) \leq 1\}$ because 
	$$\begin{aligned}
	|f_n(z)|&\leq \left\|\left(|B|_n^q\right)^{-z} w^\ast u\left(|A|_n^p\right)^z v\right\| \tau\left(|B|_n^q\right)\\
	&\leq \left\||B|_n^q\right\|^{-\Re(z)}\left\||A|_n^p\right\|^{\Re(z)} \tau\left(|B|_n^q\right)\\
	&\leq  
	\max\left\{1,\left\||B|_n^q\right\|^{-1}\right\}\max\left\{1,\left\||A|_n^p\right\|\right\} \tau\left(|B|_n^q\right).\\
	\end{aligned}$$
	
	By three-line theorem\footnote{There is a confusion concerning the name of this result. It is known as Doetsch's three-line theorem, Hadamard's three-line theorem or Phragm\'en-Lindel\"of principle. The confusion occurs because it is a variant of the three-circle theorem due to Hadamard and a consequence of Phragm\'en-Lindel\"of maximum principle, but it was published by Doetsch before Hadamard's result. }
	$$\begin{aligned}
	\left|f_n\left(\frac{1}{p}\right)\right|& \leq \sup_{y\in \mathbb{R}}|f_n(1+ \iu y)|^\frac{1}{p} \sup_{y\in \mathbb{R}}|f_n(0+\iu y)|^\frac{1}{q}\\
	&=\sup_{y\in \mathbb{R}}|\tau\left(w^\ast u\left(|A|_n^p\right)^{1+\iu y} v \left(|B|_n^q\right)^{-\iu y}\right)|^\frac{1}{p} \sup_{y\in \mathbb{R}}|\tau\left(w^\ast u\left(|A|_n^p\right)^{\iu y} v \left(|B|_n^q\right)^{(1-\iu y)}\right)|^\frac{1}{q}\\
	&\leq\sup_{y\in \mathbb{R}}\left|\tau\left(|A|_n^p\right)\right|^\frac{1}{p} \sup_{y\in \mathbb{R}}\left|\tau\left(|B|_n^q\right)\right|^\frac{1}{q}\\
	&=\tau\left(|A|_n^p\right)^\frac{1}{p} \tau\left(|B|_n^q\right)^\frac{1}{q}\\
	\end{aligned} $$
	
	Then
	$$\begin{aligned}
	\tau(|AB|)& =\lim_{n\to \infty}f_n\left(\frac{1}{p}\right)\\
	&\leq \lim_{n\to \infty}\tau\left(|A|_n^p\right)^\frac{1}{p}\tau\left(|B|_n^q\right)^\frac{1}{q}\\
	&=\tau\left(|A|^p\right)^\frac{1}{p} \tau\left(|B|^q\right)^\frac{1}{q}\\
	\end{aligned}$$
\end{proof}

We can generalise the H\"older inequality as it follows
\begin{theorem}[H\"older Inequality]
	\label{gholder}
	Let $\nalgebra$ be a von Neumann algebra and $\tau$ a normal faithful semifinite trace in $\nalgebra$, let also $A_i\in\nalgebra$, $i=1,\dots, k$ and $\displaystyle \sum_{i=1}^{k} p_i>1$ such that $\displaystyle \sum_{i=1}^{k}\frac{1}{p_i}=1$, then
	$$\tau\left(\left|\prod_{i=1}^{k}A_i\right|\right)\leq \prod_{i=1}^{k}\tau(|A_i|^{p_i})^\frac{1}{p_i}.$$
\end{theorem}
\begin{proof}
	Let $A_i=u_i|A_i|$ and $\displaystyle \prod_{i=1}^ {k}A_i=w\left|\prod_{i=1}^{k}A_i\right|$ be the respective polar decompositions of these operators. Following the same steps of the previous proof we can define the analytic several complex variables function
	
	$$\begin{aligned}
	f_n(z_1,\dots,z_{k-1})&=\tau\left(w^\ast u_1\left(|A_1|_n^{p_1}\right)^{z_1}\dots u_k \left(|A_k|_n^{p_k}\right)^{z_{k-1}}\left(|A_k|_n^{p_k}\right)^{1-\sum_{i=1}^{k-1}z_i}\right)\\
	&=\tau\left(\left(|A_k|_n^{p_k}\right)^{-\sum_{i=1}^{k-1}z_i}w^\ast u_1\left(|A_1|_n^{p_1}\right)^{z_1}\dots u_k \left(|A_k|_n^{p_k}\right)^{z_{k-1}}|A_k|_n^{p_k}\right)\\
	\end{aligned}$$
	where $\displaystyle |A_i|_n^p=\frac{1}{\tau\left(|A_i|^p\right)}\int^{\|A_i\|}_\frac{1}{n}\lambda^p dE^{|A_i|}_\lambda, \quad i=1, \dots, k$.
	
	This function is bounded in the region
	$$B=\left\{(z_1,\dots,z_{k-1}) \in \mathbb{C}^k \ \middle| \  \Re(z_i)\geq 0, \ i=1,\dots k-1, \ \sum_{i=1}^{k-1} \Re(z_i)\leq 1\right\}$$
	
	Using now the several variable three-line theorem (Theorem 2.1 in \cite{Araki73}) we have that
	$$g(y_1,\dots, y_{k-1})=\log\left(\sup_{x_1,\dots,x_{k-1} \in \mathbb{R}}\left|f_n(x_1+\iu y_1,\dots, x_{k-1}+\iu y_{k-1})\right|\right)$$
	is a convex function.
	
	This leads a equivalent result of the one we have had before, namely
	
	$$\begin{aligned}
	\left|f_n\left(\frac{1}{p_1},\dots,\frac{1}{p_{k-1}}\right)\right|& \leq \left(\sup_{y_1,\dots,y_{k-1} \in \mathbb{R}}\left|f_n(1+\iu y_1,\iu y_2,\dots,\iu y_{k-1})\right|\right)^\frac{1}{p_1}\dots\times\\
	&\qquad \times \left(\sup_{y_1,\dots,y_{k-1} \in \mathbb{R}}\left|f_n(\iu y_1,\iu y_2,\dots,1+\iu y_{k-1})\right|\right)^\frac{1}{p_{k-1}}\times\\
	&\qquad \times \left(\sup_{y_1,\dots,y_{k-1} \in \mathbb{R}}\left|f_n(\iu y_1,\iu y_2,\dots,\iu y_{k-1})\right|\right)^\frac{1}{p_{k}}\\
	&\leq \prod_{i=1}^{k}\tau(|A_i|_n^{p_i})^\frac{1}{p_i}
	\end{aligned} $$
	
	Then
	$$\begin{aligned}
	\tau(|AB|)& =\lim_{n\to \infty}f_n\left(\frac{1}{p_1},\dots,\frac{1}{p_{k-1}}\right)\\
	&\leq \lim_{n\to \infty} \prod_{i=1}^{k}\tau(|A_i|_n^{p_i})^\frac{1}{p_i}\\
	&=\prod_{i=1}^{k}\tau(|A_i|^{p_i})^\frac{1}{p_i}\\
	\end{aligned}$$
\end{proof}

\begin{theorem}[Minkowski's Inequality]
	\label{minkowski}
	Let $\nalgebra$ be a von Neumann algebra, $\tau$ a normal faithful semifinite trace in $\nalgebra$, and $p,q>1$ such that $\frac{1}{p}+\frac{1}{q}=1$, then
	\begin{enumerate}[(i)]
		\item For every $A\in \nalgebra$, $\displaystyle \tau(|A|^p)^\frac{1}{p}=\sup\left\{\left|\tau(AB)\right| \ \middle | \ B\in \nalgebra,  \tau\left(|B|^q\right)\leq 1\right\};$
		\item For every $A,B \in \nalgebra$, $\displaystyle \|A+B\|_p\leq \|A\|_p+\|B\|_p$.
	\end{enumerate}
\end{theorem}

\begin{proof}
	By H\"older's inequality, for all $B\in \nalgebra$ such that $\tau\left(|B|^q\right)\leq 1$,
	$$|\tau(AB)|\leq \tau(|AB|)\leq \tau(|A|^p)^\frac{1}{p}\tau(|B|^p)^\frac{1}{q}\leq \tau(|A|^p)^\frac{1}{p}. $$
	
	On the other hand, if $\tau(|A|^p)<\infty$, let $A=u|A|$ be the polar decomposition of $A$ and define $B=\frac{|A|^{p-1}u^\ast}{\tau\left(|A|^p\right)^\frac{p-1}{p}}$, then 
	$$|B|^2=B^\ast B=\frac{u|A|^{2p-2}u^\ast}{\tau\left(|A|^p\right)^\frac{2p-2}{p}} \Rightarrow \tau\left(|B|^q\right)=\tau\left(\frac{u|A|^{(p-1)q}u^\ast}{\tau\left(|A|^p\right)^\frac{(p-1)q}{p}}\right)=1.$$
	
	Moreover, $\tau(AB)=\frac{\tau(u|A|^p u^\ast)}{\tau\left(|A|^p\right)^\frac{p-1}{p}}=\tau\left(|A|^p\right)^\frac{1}{p}$.
	
	In the case $\tau(|A|^p)=\infty$, we use semifiniteness and normality to take an increasing sequence of operators $(A_n)_n \subset \mathfrak{M}_\tau$ converging to $A$ and apply the recently proved result to this sequence. 
	
\end{proof}

Together Theorem \ref{gholder} and Theorem \ref{minkowski} give us another little generalization of H\"older's inequality. This inequality is obvious in commutative case, but not in the noncommutative.

\begin{corollary}[H\"older Inequality]
	\label{holder}
	Let $\nalgebra$ be a von Neumann algebra and $\tau$ a normal faithful semifinite trace in $\nalgebra$, let also $A,B\in\nalgebra$ and $p,q>1$ such that $\frac{1}{p}+\frac{1}{q}=\frac{1}{r}$, then
	$$\tau(|AB|^r)^\frac{1}{r}\leq \tau(|A|^p)^\frac{1}{p}\tau(|B|^q)^\frac{1}{q}.$$
\end{corollary}
\begin{proof}
	Let $s>1$ such that $\frac{1}{r}+\frac{1}{s}=1$, then $\frac{1}{p}+{1}{q}+\left(1-{1}{r}\right)=\frac{1}{p}+{1}{q}+{1}{s}=1$. Hence using $(i)$ in Theorem \ref{minkowski} and after Theorem \ref{gholder} we get 
	$$\begin{aligned}
	\tau(|AB|^r)^{\frac{1}{r}}&=\sup\left\{\left|\tau(ABC)\right| \ \middle | \ C\in \nalgebra,  \tau\left(|C|^s\right)\leq 1\right\}\\
	&\leq \sup\left\{\tau(|A|^p)^\frac{1}{p}\tau(|B|^q)^\frac{1}{q}\tau(|C|^s)^\frac{1}{s} \ \middle | \ C\in \nalgebra,  \tau\left(|C|^s\right)\leq 1\right\}\\
	&\leq\tau(|A|^p)^\frac{1}{p}\tau(|B|^q)^\frac{1}{q}.
	\end{aligned}$$

\end{proof}

\begin{definition}
	Let $\nalgebra$ be a von Neumann algebra and $\tau$ a normal, faithful and semifinite trace on $\nalgebra$, we define the non commutative $L_p$-space, denoted by $L_p(\nalgebra,\tau)$, as the completion of
	$$\left\{A\in \nalgebra \ \middle| \ \tau\left(|A|^p\right)<\infty\right\}$$
	with respect to the norm $\displaystyle\|A\|_p=\tau\left(|A|^p\right)^\frac{1}{p}$.
	
	We also denote $L_\infty(\nalgebra, \tau)=\nalgebra$ with the norm $\|A\|_\infty=\|A\|$.
	
\end{definition}

Now, it is quite easy to see that, for $p,q\geq1$ H\"older conjugated,the H\"older and Minkowiski inequalities can be extended to the whole space $L_p(\nalgebra,\tau)$ through an argument of density  and normality of the trace. With this definition, \ref{normtraceinequality} and \ref{holder}, and $\ref{minkowski}$ can be expressed as usually
$$\begin{aligned}
\|AB\|_1&\leq \|A\|_p\|B\|_q\\
\|A+B\|_p &\leq \|A\|_p+\|B\|_p\\
\end{aligned}$$ 
and this last equality is a triangular inequality for $\|.\|_p$. It is important to notice the faithfulness guarantees $\|A\|_p=0 \Rightarrow A=0$, however semifiniteness was used only in the very end of \ref{minkowski} and it is completely unimportant when talking about non-commutative $L_p$-spaces, since the trace is never infinity on these operators.

It is not our intention in this text to discuss this subject, but notice that if $\tau$ is not semifinite we can define the noncommutative $L_p$ space to a "small" algebra $\overline{\mathfrak{M}_\tau}^{SOT}$.

\begin{theorem}
	Let $p,q\geq 1$ such that $\frac{1}{p}+\frac{1}{q}=1$, then $L_p(\nalgebra,\tau)$ and $L_q(\nalgebra,\tau)$ form a dual pair with respect to the bilinear form
	$$\begin{aligned}
	(\cdot,\cdot): &L_p(\nalgebra,\tau)\times L_q(\nalgebra,\tau)&\to&\hspace{5 mm}\mathbb{C} \\
	&\hspace{13 mm} (A,B)&\mapsto&\hspace{1 mm}\tau(AB).\\
	\end{aligned}$$
\end{theorem}
\begin{proof}
	The bilinearity is trivial and notice that, by virtue of Theorem \ref{minkowski} $(i)$, we have
	that for every non-null $A\in L_p(\nalgebra,\tau)$ there exists $B\in L_q(\nalgebra,\tau)$ such that $(A,B)\neq0$ and for every non-null $B \in L_q(\nalgebra,\tau)$ there exists $A\in L_p(\nalgebra,\tau)$ such that $(A,B)\neq 0$, which are the required properties to say that $L_p(\nalgebra,\tau)$ and $L_q(\nalgebra,\tau)$ form a dual pair under the bilinear form $(\cdot,\cdot)$.
	
\end{proof}

The last two sections of this text will be devoted to two different generalizations of $L_p$-spaces for arbitrary von Neumann algebras, one due to Haagerup and another due to Araki and Masuda.

Hitherto we have been as detailed as possible, but henceforth the purpose of the text will be slightly different: to present just an idea of how to construct a generalized $L_p$-space.

\section{Radon-Nikodym Derivative}

The fist version of an analogous of the Radon-Nikodym Theorem in noncommutative case appears in \cite{Dye52} and \cite{Segal53}. The generalization we are about to present is from \cite{Pedersen73}.

\begin{theorem}
	Let $\calgebra$ be a $C^\ast$-algebra and $B_1=\{A \in \nalgebra \ | \ \|A\|\leq1\}$ its unitary ball. Then
	$\displaystyle \mathcal{E}(B_1)=\left\{U \in \nalgebra \ \middle| \ U \textrm{ is a partial isometry and } \left(\mathbb{1}-U U^\ast\right)\calgebra\left(\mathbb{1}-U^\ast U\right)\right\}=\{0\}$.
\end{theorem}
\begin{proof}
	First of all, notice that $B_1$ is a WOT-compact set, $\mathcal{E}(B_1)\neq \emptyset$.
	
	Let $U \in \mathcal{E}(B_1)$, then $U^\ast U$ is a self-adjoint operator and $\|U^\ast U\|\leq 1$, this implies that $\sigma(U^\ast U)\subset[0,1]$. Suppose $\lambda \sigma(U^\ast U)\setminus\{0,1\}$, then there exists $\epsilon>0$ such that $(t-\epsilon,t+\epsilon)\subset[0,1]$ and we can find a positive infinitely differentiable function ${f: [0,1]\to [0,1]}$ such that $f(x)=0$ if $x \notin (t-\epsilon,t+\epsilon)$ and $0<f(\lambda)<\min\{\sqrt{\lambda},\sqrt{1-\lambda}\}$.
	
	Due to this choice $f(U^\ast U)$ commutes with $U^\ast U$ and 
	$$\begin{aligned}
	\sigma\left(U^\ast U(\mathbb{1}\pm f(U^\ast U))^2 \right)\subset [0,1]&\Rightarrow \|(\mathbb{1}\pm f(U^\ast U))U^\ast U(\mathbb{1}\pm f(U^\ast U))\|\leq 1\\
	&\Rightarrow \|U(\mathbb{1}\pm f(U^\ast U))\|\leq 1\\
	&\Rightarrow U(\mathbb{1}\pm f(U^\ast U))\in B_1,\\
	\end{aligned}$$
	but $U=\frac{1}{2}U(\mathbb{1}+ f(U^\ast U))+\frac{1}{2}U(\mathbb{1} -f(U^\ast U))$, which contradicts the extremality of $U$. Hence $\sigma(U^\ast U)\subset\{0,1\}$ and it follows $U^\ast U$ is a projection, which means $U$ is a partial isometry.
	
	Also, denote $P=U^\ast U$, $Q=UU^\ast$ and let $A\in \left(\mathbb{1}-Q\right)\calgebra\left(\mathbb{1}-P\right)$, $\|A\|\leq 1$, and $x \in \hilbert$,
	
	$$\begin{aligned}
	\|(U\pm A)x\|^2&=\|U(Px)\pm A\left((\mathbb{1}-P)x\right)\|^2\\
	&=\|QU(Px)\pm (\mathbb{1}-Q)A\left((\mathbb{1}-P)x\right)\|^2\\
	&=\|V(Px)\|^2+\|A\left((\mathbb{1}-P)x\right)\|^2\\
	&\leq \|x\|
	\end{aligned}$$
	
	Hence $U\pm A \in B_1$, and by the extremality of $U$, $$U=\frac{1}{2}(U+A)+\frac{1}{2}(U-A)\Rightarrow A=0.$$
	
	On the other hand, suppose $U$ is a partial isometry such that $\left(\mathbb{1}-Q\right)\calgebra\left(\mathbb{1}-P\right)=\{0\}$. If $U=\frac{1}{2}(A+B)$ for some $A,B\in B_1$, then, for every $x \in Ran(P)$ with $\|x\|=1$ we have $1=\ip{Ux}{Ux}=\frac{1}{2}\ip{Ax}{Ux}+\frac{1}{2}\ip{Bx}{Ux}$, but since both $\ip{Ax}{Ux}$ and $\ip{Bx}{Ux}$ are elements of the unit disk in $\mathbb{C}$ and $1$ is extremal in the disc, we have
	$$\ip{Ax}{Ux}=\ip{Bx}{Ux}=\ip{Ux}{Ux}=1\Rightarrow Ax=Bx=Ux, \forall x \in Ran(P).$$	
	
	We already have $AP=BP=U$, lets now show that $QA(1-P)=QB(1-P)=0$. Suppose it is not the case, then we can take $z\in Ran(QA(\mathbb{1}-P))\setminus\{0\}$ with $|z|$=1, which means there exists $x\in \hilbert$ such that $QA(\mathbb{1}-P)x=z$. At the same time, since $z\in Ran(Q)$ and $Q$ is the final projection of $U$, there exists $y\in \hilbert$ such that $z=Uy=APy$. Notice $1=\|z\|=\|QA(\mathbb{1}-P)x\|\leq \|(\mathbb{1}-P)x\|$ and $1=\|z\|=\|APy|=\|UPy\|=\|Py\|$.
	
	Take $\theta\in\left[0,\frac{\pi}{2}\right]$ such that $\cot(\theta)=\|(\mathbb{1}-P)x\|$ and $w=\cos(\theta)Py+\frac{\sin(\theta)}{\|(\mathbb{1}-P)x\|}(\mathbb{1}-P)x$, then
	
	$$\|QAw\|=\left|\cos(\theta)+\tan(\theta)\sin(\theta)\right|\|z\|=\left(1+\frac{1}{\|(\mathbb{1}-P)x\|^2}\right)^\frac{1}{2}>1. $$
	
	Since it is not possible once $\|QA\|\leq 1$, $QA(\mathbb{1}-P)=0$, but by hypothesis \\ ${(\mathbb{1}-Q)A(\mathbb{1}-P)=0}$, so $A(\mathbb{1}-P)=0$ and, by the same argument, $B(\mathbb{1}-P)=0$. Hence $A=B=U$ and it follows that $U$ is extremal.
	
\end{proof}

\begin{theorem}[Polar Decomposition of Functionals]
	Let $\nalgebra\subset B(\hilbert)$ be a von Neumann algebra and $\phi$ a WOT-continuous bounded functional on $\nalgebra$. Then, there exists a positive normal bounded functional $\psi$ on $\nalgebra$ and a partial isometry $U\in \nalgebra$, extreme in the unit ball, such that $\phi(A)=\psi(AU)$ and $\psi(A)=\phi(A U^\ast)$. 
\end{theorem}

\begin{proof}
	The case $\phi$=0 is trivial. Suppose $\phi\neq 0$ and let 
	$$\mathcal{F}=\{A\in\nalgebra \ | \ \|A\|\leq 1 \textrm{ and } \phi(A)=\|\phi\|\},$$
	since the unit ball of $\nalgebra$, $B_1$, is WOT-compact, there exists $\tilde{V}\in B_1$ such that $|\phi(\tilde{V})|=\|\phi\|$, hence $\phi(V)=\|\phi\|$ for $V=e^{-\iu Arg(\phi(\tilde{V}))}\tilde{V}$. Hence $\mathcal{F}\neq \emptyset$.
	
	Now, $\mathcal{F}$ is a compact face in $B_1$, thus $\mathcal{E}(\mathcal{F})\subset \mathcal{E}(B_1)$. Let $W \in \mathcal{E}(\mathcal{F})$, then $W$ is a partial isometry satisfying $(\mathbb{1}-W W^\ast)\nalgebra (\mathbb{1}-W^\ast W)=\{0\}$.
	
	Define $\psi$ by $\psi(A)=\phi(AW)$ for every $A\in \nalgebra$. Then $$|\psi(A)|=\phi(AW)\leq \|\phi\|\|A\|\|W\|=\|\phi\|\|A\|\Rightarrow \|\psi\|\leq \|\phi\|,$$
	thus $\psi(\mathbb{1})=\phi(W)=\|\phi\|=\|\psi\|$, it follows by Proposition \ref{PPC} that $\psi$ is positive.
	
	Let $s^\nalgebra(\psi)$ be the support of $\psi$, then $s^\nalgebra(\psi)\leq W W^\ast$, since $$\psi(W W^\ast)=\phi(WW^\ast W)=\phi(W)=\psi(\mathbb{1}) \Rightarrow \mathbb{1}-WW^\ast\in N_{\psi}$$ and $N_{\psi}$ is a left ideal.
	
	Let $U=W^\ast s^\nalgebra(\psi)$, we have $U^\ast U=s^\nalgebra(\psi)W W^\ast s^\nalgebra(\psi)=s^\nalgebra(\psi)$, then
	$$\phi(AU^\ast)=\phi\left(A s^\nalgebra(\psi) W\right)=\psi\left(As^\nalgebra(\psi)\right)=\psi(A).$$
	
	Suppose now that there exists $A\in\nalgebra$, $\|A\|=1$, such that 
	$\phi\left(A(\mathbb{1}-U U^\ast )\right)>0 $ then, for every $t \in \left[0,\frac{\pi}{2}\right]$ such that $\cot\left(\frac{t}{2}\right)>\frac{\|\phi\|}{\phi\left(A(\mathbb{1}-U U^\ast )\right)}$, we have
	$$\begin{aligned}
	\phi\big(\cos(t) U^\ast +\sin(t) A(\mathbb{1}-U U^\ast )\big)
	&=\cos(t)\phi(U^\ast)+\sin(t)\phi\big(A(\mathbb{1}-U U^\ast )\big)\\
	&=\cos(t)\phi(W)+\sin(t)\phi\left(A(\mathbb{1}-U U^\ast )\right)\\ 
	&=\cos(t)\|\phi\|+\sin(t)\phi\left(A(\mathbb{1}-U U^\ast )\right)\\ &>\|\phi\|.\end{aligned}$$
	
	If follows that $\phi\left(A(\mathbb{1}-U U^\ast )\right)=0\Rightarrow \phi(A)=\phi(AUU^\ast)=\psi(AU)$.
	
\end{proof}

\begin{notation}
	The $\psi$ obtained in the previous theorem is called the modulus of $\phi$ and usually denoted by $|\phi|$, also, the polar decomposition of $\phi$ written as $\phi=\hat{U}|\phi|$.
\end{notation}

It may worth to mention that the modulus of an normal function is unique, and, if we require $U^\ast U =s^\nalgebra(\phi)$, so is $U$.

\begin{proposition}
	\label{selfadjointfunctional}
	Let $\nalgebra$ be a von Neumann algebra and $\phi$ a positive linear functional on $\nalgebra$ and for $H\in\nalgebra$ define $\left(\hat{H}\right)\phi(A)=\phi(AH)$. Then, if $\hat{H}\phi$ is self-adjoint, $$\left|\left(\hat{H}\phi(A)\right)\right|=|\phi(AH)|\leq \|H\|\phi(A), \ \forall A\in\nalgebra.$$
\end{proposition}
\begin{proof}
	By self-disjointness,
	$$\phi(AH)=\hat{H}\phi(A)=(\hat{H}\phi)^\ast(A)=\overline{\hat{H}\phi(A^\ast)}= \overline{\phi(A^\ast H)}=\phi(H^\ast A).$$
	
	Hence $\phi\left(A H^{2^{n+1}}\right)=\phi\left((H^{2^n})^\ast AH^{2^n}\right)$, then, for every $A\geq 0$,
	$$\begin{aligned}
	|\phi(AH)|
	&=\left|\phi(A^\frac{1}{2}A^\frac{1}{2}H)\right|\\
	&\leq \left|\phi(A)\right|^\frac{1}{2}\left|\phi(H^\ast A H)\right|^\frac{1}{2}\\
	&= \left|\phi(A)\right|^\frac{1}{2}\left|\phi(A H^2)\right|^\frac{1}{2}\\
	&\leq \left|\phi(A)\right|^{\sum_{i=1}^{n} 2^{-i}}\left|\phi(A H^{2^n})\right|^\frac{1}{2^n}\\
	&=\left|\phi(A)\right|^{1-2^{-n}}\left|\phi(A H^{2^n})\right|^\frac{1}{2^n}\\
	&\leq\left|\phi(A)\right|^{1-2^{-n}}\left(\|\phi\|\|A\|\|H\|^{2^n}\right)^\frac{1}{2^n}\\
	&\xrightarrow{n\to\infty} \|H\|\left|\phi(A)\right|.\\
	\end{aligned}$$
	
\end{proof}

\begin{proposition}
	\label{commutantRN}
	Let $\nalgebra$ be a von Neumann algebra, $\phi, \psi$ be normal semi-finite weights on $\nalgebra$ such that $\psi\leq \phi$. Then, there exists an operator $H^\prime \in \nalgebra^\prime$, $0\leq H\leq 1$, such that
	$$\psi(A)=\ip{H \pi_\phi(A^\ast) \Phi}{\Phi}_\phi, \quad \forall A\in\nalgebra.$$
\end{proposition}
\begin{proof}
	We will use the GNS-representation throughout the weight $\phi$. Notice that, by hypothesis, $\mathfrak{N}_\phi\subset \mathfrak{N}_\psi$ and $N_\phi\subset N_\psi$.
	
	Define the sesquilinear form $\tilde{\psi}: \mathfrak{N}_\phi/N_\phi\times \mathfrak{N}_\phi/N_\phi \to \mathbb{C}$ given by $\tilde{\psi}([A],[B])=\psi(B A^\ast)$ which is well defined by the same calculation presented in Equation \eqref{eq:calculationquotient}.
	
	By Cauchy-Schwarz's inequality and the hypothesis, $$|\tilde{\psi}([A],[B])|^2=|\psi(BA^\ast )|^2\leq \psi(A^\ast A) \psi(B^\ast B)\leq \phi(A^\ast A) \phi(B^\ast B)=\|[A]\|_\phi^2 \|[B]\|_\phi^2.$$
	
	Hence $\tilde{\psi}$ admits a unique extension to a sesquilinear form on $\hilbert_\phi$, also denoted by $\tilde{\psi}$. By the Riesz Theorem, there exists a unique operator $H^\prime\in B(\hilbert_\phi)$ such that
	\begin{equation}
	\label{eq:riesz1}
	\tilde{\psi}(x,y)=\ip{H^\prime x}{y}_{\phi} \quad \forall x,y \in \hilbert_\phi,
	\end{equation}
	it also follows by the theorem that $\|H^\prime \|\leq 1$ and by positiveness of $\tilde{\psi}$, $H^\prime \geq 0$.
	
	In addition, for every $A,B\in\mathfrak{N}_\phi$ and $C\in\nalgebra$ we have
	$$\begin{aligned}
	\ip{(H^\prime \pi_\phi(C)-\pi_\phi(C)H^\prime )[A]}{[B]}_\phi&=\ip{H^\prime \pi_\phi(C)[A]}{[B]}_\phi-\ip{H^\prime [A]}{\pi_\phi(C)^\ast [B]}_\phi\\
	&=\ip{H^\prime [CA]}{[B]}_\phi-\ip{H^\prime [A]}{[C^\ast B]}_\phi\\
	&=\tilde{\psi}\left((CA)^\ast B\right)-\tilde{\psi}\left(A^\ast C^\ast B\right)\\
	&=0
	\end{aligned}$$
	thus $(H^\prime \pi_\phi(C)-\pi_\phi(C)H^\prime )[A]=0 \quad \forall A\in \mathfrak{N}_\phi \Rightarrow H^\prime \pi_\phi(C)=\pi_\phi(C)H^\prime  \Rightarrow H^\prime \in \nalgebra^\prime$.
	
	Finally, Equation \eqref{eq:riesz1}, in the special case $(E_\alpha)_\alpha \subset \mathfrak{N}_\phi$ is an approximation identity and $A\in \nalgebra_+$, can be rewritten as
	
	$$\psi(E_\alpha A E_\alpha)=\tilde{\psi}(\pi_\phi(A^\ast) [E_\alpha],[E_\alpha])=\ip{H^\prime \pi_\phi(A^\ast)[E_\alpha]}{[E_\alpha]}_{\phi}$$
	and the thesis follows by normality and the polarization identity.
	
\end{proof}
\begin{remark}
	It is common in the literature to express the previous proposition as
	$$\exists H^\prime \in \nalgebra^\prime ; \ \psi(A)=\phi(H^\prime A), \quad \forall A\in \nalgebra_+.$$
	
	It is important to stress that this is a abuse of notation, because there is no reason to $H^\prime A \in \nalgebra$, once $H^\prime \in \nalgebra^\prime$. Nevertheless, the expression makes sense if it is seen as an extension of $\phi$ in $B(\hilbert)$
\end{remark}

\begin{theorem}[Sakai-Radon-Nikodym]
	\label{TSRN}
	Let $\nalgebra$ be a von Neumann algebra, $\phi, \psi$ be normal functionals on $\nalgebra$ such that $\psi\leq \phi$. Then, there exists an operator $H\in \nalgebra$, $0\leq H\leq 1$, such that
	$$\psi(A)=\phi(HAH), \quad \forall A\in \nalgebra_+.$$
\end{theorem}

\begin{proof}
	
	By Theorem \ref{commutantRN}, there exists $H^\prime\in\nalgebra^\prime$, $0\leq H^\prime\leq \mathbb{1}$ such that
	$$\psi(A)=\ip{A H^\prime\Phi}{H^\prime\Phi}_\phi, \ \forall A\in\nalgebra_+.$$
	
	Consider the WOT-continuous functional $\phi^\prime: \nalgebra^\prime \to \mathbb{C}$ given by $\phi^\prime(A^\prime)=\ip{A^\prime H^\prime \Phi}{\Phi}_\phi$.
	
	Let $|\phi^\prime|=\widehat{U^{\prime\ast}}\phi^\prime$ be the polar decomposition of $\phi^\prime$, then $$|\phi^\prime|=\widehat{U^{\prime\ast} H^\prime}\omega,$$  where ${\omega(A^\prime)=\ip{A^\prime\Phi}{\Phi}_\phi}$ and, by Proposition \ref{selfadjointfunctional} $$|\phi^\prime|(A)=\widehat{U^{\prime\ast} H^\prime}\omega\leq \|U^\ast H^\prime\| \omega\leq \|U^{\prime\ast}\|\|H^\prime\|\omega\leq \omega.$$
	
	Using now the Proposition \ref{commutantRN} for $|\phi^\prime|$ and $\omega$, there exists $H\in \nalgebra^{\prime\prime}=\nalgebra$, $0\leq H\leq \mathbb{1}$ such that
	$$|\phi^\prime|(A^\prime)=\ip{A^\prime H\Phi}{\Phi}_\phi.$$
	
	Now we have all the elements we will need to conclude the proof, just remains to do some calculations.
	
	Notice that, for every $A^\prime \in \nalgebra^\prime, $ $$\begin{aligned}
	\ip{H^\prime \Phi}{A^{\prime}\Phi}_\phi&=\ip{A^{\prime\ast} H^\prime \Phi}{\Phi}_\phi\\
	&=\phi^\prime(A^{\prime\ast})\\
	&=\widehat{U^\prime}|\phi^\prime|(A^{\prime\ast})\\
	&=\widehat{U^\prime}\ip{A^{\prime\ast} H \Phi}{\Phi}_\phi\\
	&=\ip{A^{\prime\ast} H U^\prime\Phi}{\Phi}_\phi\\
	&=\ip{U^\prime H \Phi}{A^{\prime}\Phi}_\phi\\
	\end{aligned}$$	from which, since $\nalgebra^\prime\Phi$ is dense in $\hilbert$, we conclude $H^\prime\Phi=U^\prime H\Phi$.
	
	Moreover, for every $A^\prime \in\nalgebra\prime$,
	$$\begin{aligned}
	\ip{ H\Phi}{A^{\prime}\Phi}_\phi&=\ip{A^{\prime\ast} H\Phi}{\Phi}_\phi\\
	&=|\phi^\prime|(A^{\prime\ast})\\
	&=\widehat{U^{\prime\ast}}\phi^\prime(A^{\prime\ast})\\
	&=\ip{A^{\prime\ast} H^\prime U^{\prime\ast}\Phi}{\Phi}_\phi\\
	&=\ip{ U^{\prime\ast}H^\prime\Phi}{A^{\prime}\Phi}_\phi\\
	\end{aligned}$$
	and, again by the denseness argument, we conclude that $H\Phi=U^{\prime\ast}H^\prime\Phi$.
	
	Finally, 
	$$\begin{aligned}
	\psi(A)&=\ip{A H^\prime \Phi}{H^\prime\Phi}_\phi\\
	&=\ip{A H^\prime \Phi}{U^\prime H\Phi}_\phi\\
	&=\ip{HA U^{\prime\ast}H^\prime \Phi}{\Phi}_\phi\\
	&=\ip{HA H \Phi}{\Phi}_\phi\\
	&=\phi(HAH), \quad \forall A\in\nalgebra_+.\\
	\end{aligned}$$

\end{proof}

Using the GNS-representation we can generalise Theorem \ref{TSRN} for weights.

\begin{theorem}[Sakai-Radon-Nikodym for Weights]
	\label{TSRNweights}
	Let $\nalgebra$ be a von Neumann algebra, $\phi, \psi$ be normal semi-finite weights on $\nalgebra$ such that $\psi\leq \phi$. Then, there exists an operator $H\in \nalgebra$, $0\leq H\leq 1$, such that
	$$\psi(A)=\phi(HAH), \quad \forall A\in \nalgebra_+.$$
\end{theorem}
\begin{proof}
	As in Proposition \ref{commutantRN} there exists a unique sesquilinear form $\tilde{\psi}:\hilbert_\phi \times \hilbert_\phi\to \mathbb{C}$ such that $\tilde{\psi}([A],[B])=\psi(A^\ast B)$, for all $A,B\in\mathfrak{N}_\phi$.
	
	Then, by Theorem \ref{TSRN}, there exists a unique $H\in\pi_\phi\left(\nalgebra\right)_+$, $0\leq H\leq \mathbb{1}$, such that $\phi(A^\ast A)=\tilde{\psi}([A],[A])=\ip{A^\ast A H\Phi}{H\Phi}=\phi(\pi^{-1}_\phi(H)A^\ast A \pi^{-1}_\phi(H)), \ A\in\mathfrak{N}_+$, and the thesis follows by semi-finiteness and normality.
	
\end{proof}

\begin{definition}
	Let $\calgebra$ be a $C^\ast$-algebra, $\{\sigma_t\}_{t\in\mathbb{R}}$ and a one parameter group of automorphisms of $\calgebra$. An lower semi-continuous weight $\phi$ is said to satisfy the modular condition for $\{\sigma_t\}_{t\in\mathbb{R}}$ if:
	\begin{enumerate}[(i)]
		\item $\phi=\phi\circ\sigma_t$, for every $t\in\mathbb{R}$;
		\item for every $A,B\in \mathfrak{N}_\phi\cap \mathfrak{N}_\phi^\ast$, there exists a complex function $F_{A,B}$ which is analytic in the strip $\left\{z\in \mathbb{C} \mid 0< \Im z<1\right\}$ and continuous and bounded on its closure satisfying
		\begin{equation}
		\label{eq:ModCond}
		\begin{aligned}
		F_{A,B}(t) &= \phi(A\sigma_t(B)), \ \forall t\in \mathbb{R}, \\
		F_{A,B}(t+\iu) &= \phi(\sigma_t(B)A), \ \forall t\in \mathbb{R}. \\
		\end{aligned}
		\end{equation}
	\end{enumerate}
\end{definition}

It becomes evident that the modular condition is the KMS-condition for $\beta=-1$.

\begin{proposition}
	\label{twosidedmodule}
	Let $\phi$ be a normal semifinite weight on a von Neumann algebra $\nalgebra$, $\{\tau^\phi_t\}_{t\in \mathbb{R}}$ its modular automorphism group and $\nanalytic$ the set of all entire $\{\tau^\phi_t\}_{t\in \mathbb{R}}$-analytic elements of $\nalgebra$. Then
	\begin{enumerate}[(i)]
		\item $\mathfrak{N}_\phi\cap \mathfrak{N}_\phi^\ast$ is a two-sided module over $\nanalytic$; 
		\item $\mathfrak{M}_\phi$ is a two-sided module over $\nanalytic$.
	\end{enumerate}
\end{proposition}
\begin{proof}
	In this proof we will refer to the items of Proposition \ref{simplepropweights}.
	
	$(i)$ The proof that $\mathfrak{N}_\phi\cap \mathfrak{N}_\phi^\ast$ is a $\ast$-subalgebra is basically already in Proposition \ref{simplepropweights}. So it remains to show the module property.
	
	Let $N\in \mathfrak{N}_\phi\cap \mathfrak{N}_\phi^\ast$ and $A\in\nanalytic$, of course $AN\in\mathfrak{N}_\phi$, since $\mathfrak{N}_\phi$ is a left ideal as stated in item $(iii)$. Remains to show $AN\in\mathfrak{N}_\phi^\ast$.
	
	It follows from the modular condition that there exists a complex function $F_{ANN^\ast,A^\ast}$ which is analytic in the strip $\left\{z\in \mathbb{C} \mid 0< \Im z<1\right\}$ and continuous and bounded on its closure satisfying 
	$$\begin{alignedat}{2}
	F_{ANN^\ast,A^\ast}(t) &= \phi\left(ANN^\ast\tau^\phi_t(A^\ast)\right), \ \forall t\in \mathbb{R}, \\
	F_{ANN^\ast,A^\ast}(t+\iu)&= \phi\left(\tau^\phi_t(A^\ast)ANN^\ast\right), \ \forall t\in \mathbb{R}. \\
	\end{alignedat}$$
	Since $\tau^\phi_t(A^\ast)=\tau^\phi_t(A)^\ast$, $A^\ast\in\nanalytic$ and
	$$\begin{aligned}
	\phi\left(ANN^\ast\tau^\phi_t(A^\ast)\right)&=F_{ANN^\ast,A^\ast}(t)\\
	&=\phi\left(\tau^\phi_{-\iu}(A^\ast)ANN^\ast\right)\\
	&\leq \phi\left(N^\ast A^\ast\tau^\phi_{-\iu}(A^\ast)^\ast\tau^\phi_{-\iu}(A^\ast)A N\right)^\frac{1}{2}\phi(N N^\ast)^\frac{1}{2}\\
	&\leq\|A\|\left\|\tau^\phi_{-\iu}(A^\ast)\right\|\phi(N^\ast N)^\frac{1}{2}\phi(N N^\ast)^\frac{1}{2}\\
	&<\infty.
	\end{aligned}$$
	
	$NA\in \mathfrak{N}_\phi\cap \mathfrak{N}_\phi^\ast$ follows by the same argument just noticing that $NA=(A^\ast N^\ast)^\ast$.
	
	$(ii)$ Again, it is already done in item $(v)$ that $\mathfrak{M}_\phi$ is a $\ast$-subalgebra.
	
	By definition $\mathfrak{M}_\phi=span\left[\mathfrak{F}_\phi\right]$ and by item $(i)$ $\mathfrak{F}_\phi\subset \mathfrak{N}_\phi\cap \mathfrak{N}_\phi^\ast$. The conclusion is now obvious.
	
\end{proof}

\begin{theorem}
	\label{UniqueModCond}
	Let $\nalgebra$ be a von Neumann algebra and $\phi$ a faithful normal semifinite weight on $\nalgebra$. Then, there exists a unique one parameter group of automorphisms $\{\tau_t\}_{t\in\mathbb{R}}$ satisfying the modular condition.
\end{theorem}

\begin{corollary}
	\label{modularautomosphism}
	Let $\nalgebra_1, \nalgebra_2$ be von Neumann algebras, $\phi$ be a normal semifinite weight on $\nalgebra_1$, and $\pi:\nalgebra_1\to\nalgebra_2$ an isomorphism. Then the modular automorphism group of $\phi\circ\phi$ is $\{\pi^{-1}\circ\tau_t^\phi\circ\pi\}_{t\in\mathbb{R}}$.
\end{corollary}
\begin{proof}
	Lets proof its satisfy the two conditions of the definition.
	
	$(i)$ $(\phi\circ\pi)\circ(\pi{-1}\circ\tau_t\circ\pi)=\phi\circ\tau_t\circ\pi=\phi\circ\pi$, where we use that $\phi\circ\tau_t=\phi$.
	
	$(ii)$ For every $A,B\in\mathfrak{N}_{\phi}\cap \mathfrak{N}_{\phi}^\ast$ there exists a complex function $F_{A,B}$ which is analytic in the strip $\left\{z\in \mathbb{C} \mid 0< \Im z<1\right\}$ and continuous and bounded on its closure satisfying
	$$\begin{aligned}
	F_{A,B}(t) &=& \phi(A\sigma_t(B))&= \phi\circ\pi\left(\pi^{-1}(A)\pi^{-1}\circ\sigma_t\circ\pi\left(\pi^{-1}(B)\right)\right) \ \forall t\in \mathbb{R} \\
	F_{A,B}(t+\iu) & =& \phi(\sigma_t(A)B)&= \phi\circ\pi\left(\pi^{-1}\circ\sigma_t\circ\pi(B)\pi^{-1}(A)\right), \ \forall t\in \mathbb{R}. \\
	\end{aligned}$$
	
	Notice now that $\mathfrak{N}_{\phi\circ\pi}\cap \mathfrak{N}_{\phi\circ\pi}^\ast=\pi^{-1}\left(\mathfrak{N}_\phi\right)\cap\pi^{-1}\left(\mathfrak{N}_\phi^\ast\right)$, then the thesis is clear choosing $F_{\pi^{-1}(A),\pi^{-1}(B)}=F_{A,B}$.
	
\end{proof}

\begin{definition}[Centralizer of a Weight]
	Let $\nalgebra$ be a von Neumann algebra, $\phi$ a faithful normal semifinite weight on $\nalgebra$, and $\tau^\phi=\{\tau^\phi_t\}_{t\in\mathbb{R}}$ the modular automorphism group of $\phi$, we define the \index{weight! centrilizer} the centralizer of $\phi$ as the set
	$$\mathfrak{M}_{\tau^\phi}=\left\{A\in \nalgebra \ | \ \tau^\phi_t(A)=A, \ \forall t\in\mathbb{R}\right\}.$$
\end{definition}

Notice that it follows from linearity and normality that $\mathfrak{M}_{\tau^\phi}$ is a von Neumann sub algebra of $\nalgebra$.

\begin{theorem}
	\label{centralizercommute}
	Let $\nalgebra$ be a von Neumann algebra, $\phi$ a faithful normal semifinite weight. Then, $A\in\mathfrak{M}_{\tau\phi}$ if, and only if
	\begin{enumerate}[(i)]
		\item $A\mathfrak{M}_\phi \subset \mathfrak{M}_\phi$ and $\mathfrak{M}_\phi A\subset \mathfrak{M}_\phi$;
		\item $\phi(AB)=\phi(BA)$, $\forall B\in\mathfrak{M}_\phi$.
	\end{enumerate}
\end{theorem}
\begin{proof}
	Let's denote $\tau^\phi=\{\tau^\phi_t\}_{t\in\mathbb{R}}$ the modular automorphism group of $\phi$.
	
	$(\Rightarrow)$ Since $A\in\mathfrak{M}_{\tau^\phi}$, $A$ is an entire analytic element. Then, condition $(i)$ follows from $(ii)$ in Proposition \ref{twosidedmodule}.
	
	Since $B\in\mathfrak{M}_\phi$, $B= C^\ast D$ with $C,D \in \mathfrak{N}_\phi$. By Proposition \ref{twosidedmodule} and modular condition there exists analytic functions on the strip $\strip{1}$, continuous and bounded on its closure, such that
	$$\begin{alignedat}{4}
	F_{C^\ast,DA}(t) &=\phi\left(C^\ast\tau^\phi_t(D A)\right) &=\phi\left(C^\ast\tau^\phi_{t}(D)A\right),& \\
	F_{C^\ast,DA}(t+\iu) &= \phi\left(\tau^\phi_t(DA)C^\ast\right) &=\phi\left(\tau^\phi_{t}(D)AC^\ast\right) &=\phi\left(\tau^\phi_{t}(D)AC^\ast\right), \\
	\end{alignedat} \quad \forall t\in \mathbb{R}$$
	and
	$$\begin{alignedat}{3}
	F_{\tau^\phi_{t}(D),A\tau^\phi_{-t}(C^\ast)}(t)
	&=\phi\left(\tau^\phi_{t}(D)\tau^\phi_{t}\left(A\tau^\phi_{-t}(C^\ast)\right)\right) 
	&=\phi\left((\tau^\phi_{t}(D)AC^\ast\right), \\
	F_{\tau^\phi_{t}(D),A\tau^\phi_{-t}(C^\ast)}(t+\iu) 
	&=\phi\left(\tau^\phi_t\left(A\tau^\phi_{-t}(C^\ast)\right) \tau^\phi_{t}(D)\right)
	&= \phi\left(AC^\ast \tau^\phi_{t}(D)\right), \\
	\end{alignedat} 	\quad \forall t\in \mathbb{R}.$$

	Now, by the previous equation and by $A\in\mathfrak{M}_\phi$ we have $F_{C^\ast,DA}(t+\iu)=F_{\tau^\phi_{t}(D),AC^\ast}(t)$ and $F_{C^\ast,DA}(t)=F_{\tau^\phi_{t}(D),AC^\ast}(t+\iu)$, so we can define the bounded function $G:\mathbb{C} \to \mathbb{C}$ by
	$$G(z)=\begin{cases}F_{C^\ast,DA}(z-2n\iu) & \textrm{if } 2n\leq Im(z)\leq 2n+1, \ n\in\mathbb{Z}\\ F_{\tau^\phi_{t}(D),AC^\ast}(z-(2n+1)\iu) & \textrm{if } 2n+1\leq Im(z)\leq 2n+2, n\in\mathbb{Z}.\\\end{cases}$$
	
	It follows from the edge-of-the-wedge theorem that $G$ entire analytic, but its also bounded, thus its constant by Liouville's theorem. Hence
	$$\phi(B A)=F_{C^\ast,DA}(0) =G(0)=G(2\iu)= F_{\tau^\phi_{t}(D),AC^\ast}(\iu)=\phi(AB), \ \forall B\in\mathfrak{M}_{\tau^\phi}.$$
	
	$(\Leftarrow)$ Notice that the assumptions $(i)$ warrants that, for any $B\in\mathfrak{N}_\phi$, we can again define the analytic function on $\strip{1}$ which is continuous and bounded on $\overline{\strip{1}}$,
	$F(t)=\phi\left(\tau^\phi_t(A)B\right)=\phi\circ\tau^\phi_t\left(A\tau^\phi_{-t}(B)\right)$.
	
	Form assumption $(ii)$, this function is periodic with period $\iu$, then, using a definition similar to the one we used to $G$ above, $F$ can be analytically extended to the whole complex plane by the edge-of-the-wedge theorem. Since the extension still bounded,  by Liouville's theorem, $F$ is constant.
	
	Hence, $\phi\left(\tau^\phi_t(A)B\right)=F(t)=F(0)=\phi(AB), \ \forall B\in\mathfrak{N}_\phi\Rightarrow \tau^\phi_t(A)=A$.
	
\end{proof}

\begin{lemma}
	\label{centralizerorder}
	Let $\nalgebra$ be a von Neumann algebra, $\phi$ a faithful normal seimifinite weight on $\nalgebra$,
	$$\begin{aligned}
	\mathfrak{M}_{\tau^\phi}\ni H \mapsto & \ \phi_H&: \ &\nalgebra^+ &\to &\hspace{0.9cm}\overline{\mathbb{R}}\\
	& \ & \ &A& \ & \phi\left(H^\frac{1}{2}AH^\frac{1}{2}\right)\\
	\end{aligned}$$
	is a order preserving map on the weights on $\nalgebra$.
\end{lemma}
\begin{proof}
	Normality is evident and semifiniteness follows from Proposition \ref{twosidedmodule}. To the order, notice that $H, K\mathfrak{M}_\phi^+$ implies $H^\frac{1}{2}, K^\frac{1}{2}, (H+K)^\frac{1}{2}\in \mathfrak{M}_\phi^+$, thus $$\phi\left(H^\frac{1}{2}A H^\frac{1}{2}\right)=\phi(HA) \textrm{ and } \phi\left(K^\frac{1}{2}A K^\frac{1}{2}\right)=\phi(KA).$$
	
	Hence, if $K\leq H$, 
	$$\begin{aligned}
	\phi(H^\frac{1}{2}A H^\frac{1}{2})\\
	&=\phi(HA)\\
	&=\phi\left(((H-K)+K)A\right)\\
	&=\phi\left(((H-K)A\right)+\phi(KA)\\
	&=\phi\left((H-K)^\frac{1}{2}A(H-K)^\frac{1}{2}\right)+\phi\left(K^\frac{1}{2}AK^\frac{1}{2}\right)\\
	&\geq \phi\left(K^\frac{1}{2}AK^\frac{1}{2}\right).\\
	\end{aligned}$$
\end{proof}
\begin{lemma}
	\label{unboundedderivativeweight}
	Let $\nalgebra$ be a von Neumann algebra, $\phi$ a faithful normal seimifinite weight on $\nalgebra$ and $H\eta \mathfrak{M}_{\tau^\phi}^+$.  If $\left(H_i\right)_{i\in I} \in  \mathfrak{M}_{\tau^\phi}^+$ is an increasing net such that $H_i \to H$, then
	\begin{equation}
	\label{eq:definitionx1}
	\phi_H(A)=\lim_{H_i \to H}\phi\left(H_i A\right)=\lim_{H_i \to H}\phi\left(H_i^\frac{1}{2}AH_i^\frac{1}{2}\right)=\sup_{i\in I}{\phi\left(H_i^\frac{1}{2}AH_i^\frac{1}{2}\right)}, \ \forall A\in\nalgebra
	\end{equation}
	defines a normal semifinite weight $\phi_H$ on $\nalgebra$ which is independent of the choice of $(H_i)_{i\in I}$.	
	
	In addition, $\phi_H$ is faithful if, and only if, $H$ is non-singular and, if $(H_i)_{i\in I}$ is an increasing  net of positive operators affiliated with $\mathfrak{M}_{\tau^\phi}^+$ such that $H_i\to H$, then 
	$$\phi_H=\sup_{i\in I}\phi_{H_i}.$$
	
\end{lemma}
\begin{proof}
	By Lemma \ref{centralizerorder}, $\phi_H(A)$ well is defined since it is the limit of a positive increasing net of real numbers. Furthermore, it is easy to see (by the same Lemma and normality of $\phi$) that it is a normal weight. Remains to prove it is semifinite.
	
	Let $\left\{E^{H}_\lambda\right\}_{\lambda\in\mathbb{R}_+}$ is the spectral decomposition of $H$. By Proposition \ref{twosidedmodule} $(ii$) $\displaystyle \bigcup_{n\in\mathbb{N}}E_{[0,n)}\mathfrak{M}_\phi E_{[0,n)}\in\mathfrak{M}_\phi$ and it is WOT-dense in $\nalgebra$.
	
	To prove the independence of the net, let $H_n=E^H_{[0,n)}H E^H_{[0,n)}$ be fixed and let $(K_j)_{j\in J}\in \mathfrak{M}_{\tau^\phi}$ be another increasing net such that $K_j \to H$. Denote by $\phi_H$ the normal semifinite weight defined by Equation \eqref{eq:definitionx1} for the sequence $(H_n)_n$.
	
	We know that $K_{j,n}=E^{H}_{[0,n)}K_jE^{H}_{[0,n)}$ is an increasing net with $\displaystyle \sup_{j\in J}K_{j,n}=H_n$ and $\displaystyle \sup_{n\in \mathbb{N}}K_{j,n}=K_j$. Let's use the GNS-representation throughout $\phi$. Notice that, $\pi_\phi(K_{j,n})\xrightarrow[SOT]{j} \pi_\phi(H_n)$ and $\pi_\phi(K_{j,n})\xrightarrow[SOT]{n} \pi_\phi(K_j)$ due to Vigier's theorem, then
	$$\begin{aligned}
	&\sup_{j\in J}\phi(K_{n,j}A)=\sup_{j\in J}\ip{\pi_{\phi}(K_{j,n})\Phi}{\pi_{\phi}(A)\Phi}_\phi=\ip{\pi_{\phi}(H_n)\Phi}{\pi_{\phi}(A)\Phi}_\phi=\phi(H_nA)\\
	&\sup_{n\in \mathbb{N}}\phi(K_{n,j}A)=\sup_{n\in \mathbb{N}}\ip{\pi_{\phi}(K_{j,n})\Phi}{\pi_{\phi}(A)\Phi}_\phi=\ip{\pi_{\phi}(K_j)\Phi}{\pi_{\phi}(A)\Phi}_\phi=\phi(K_jA)\\
	\end{aligned}$$
	Hence
	
	$$\begin{aligned}
	\label{eq:calculationX17}
	\phi_H(A)&=\sup_{n\in \mathbb{N}}{\phi\left(H_n^\frac{1}{2}AH_n^\frac{1}{2}\right)}\\
	&=\sup_{n\in \mathbb{N}}\sup_{j\in J}{\phi\left(K_{n,j}^\frac{1}{2}AK_{n,j}^\frac{1}{2}\right)}\\
	&\leq \sup_{j\in J}{\phi\left(K_j^\frac{1}{2}AK_j^\frac{1}{2}\right)}\\
	&=\sup_{j\in J}\sup_{n\in \mathbb{N}}{\phi\left(K_{n,j}^\frac{1}{2}AK_{n,j}^\frac{1}{2}\right)}\\
	&=\phi_H(A).\\
	\end{aligned}$$
	
	For the last statement, let $(H_i)_{i\in I}$ an increasing net of operators affiliated with $\mathfrak{M}_{\tau^\phi}$ such that $H_i\to H$, then we can define $H_{i,n}=E^{H_i}_{[0,n)} H_i E^{H_i}_{[0,n)}$ and using what we get in Equation \eqref{eq:calculationX17} to obtain
	$$\begin{aligned}
	\phi_{H}=\sup_{i\in I}\phi_{H_{i,n}}=\sup_{i\in I}\phi_{H_i}.
	\end{aligned}$$
\end{proof}

\begin{notation}
	Henceforth, when $H$ is a positive unbounded operator affiliated with $\mathfrak{M}_{\tau^\phi}^+$ we will consider the the weight $\phi_H$ and usually write $\phi(HA)$ instead of $\phi_H(A)$.
	
\end{notation}

\begin{theorem}[Pedersen-Takesaki-Radon-Nikodym]
	\label{TPTRN}
	Let $\phi$ and $\psi$ be two normal semifinite weights on a von Neumann algebra $\nalgebra$. Suppose in addition that $\phi$ is faithful and $\psi$ is invariant under the modular automorphism group of $\phi$, $\{\tau^\phi_t\}_{t\in \mathbb{R}}$, and $\psi\leq \phi$. Then, there exists a unique operator $H\in \mathfrak{M}_{\tau^\phi}$ with $0\leq H\leq \mathbb 1$ and such that $\psi(A)=\phi(HA)$, $\forall A\in \nalgebra_+$.
\end{theorem}
\begin{proof}
	
	By Proposition \ref{commutantRN}, there exists $H^\prime \in \nalgebra^\prime$, $0\leq H^\prime \leq \mathbb{1}$, such that $$ \psi(A)=\ip{H^\prime A^\ast \Phi}{\Phi}_\phi, \quad \forall A\in\nalgebra_+.$$
	
	Notice that the invariance of $\psi$ implies
	$$\psi\left(\tau^\phi_{t}(A)\right)=\ip{H^\prime\tau^\phi_{t}(A)\Phi}{\Phi}_\phi=\ip{\tau^\phi_{t}(H^\prime)A\Phi}{\Phi}_\phi=\ip{H^\prime A\Phi}{\Phi}_\phi=\psi(A), \ \ \forall A\in \nalgebra_\phi, $$	thus $H^\prime$ is invariant under $\{\tau^\phi_t\}_{t\in \mathbb{R}}$.
	
	By Tomita's theorem we have that $H=J_\phi H^\prime J_\phi \in \nalgebra$ and it is also $\{\tau^\phi_t\}_{t\in \mathbb{R}}$-invariant. Furthermore, for every $A\in \mathfrak{N}_\phi\cap \mathfrak{N}_\phi^\ast$,
	
	\begin{equation}
	\label{eq:calculationRN}
	\begin{aligned}
	\phi\left(H^\frac{1}{2}\tau^\phi_{z-\frac{\iu}{2}}(A)H^\frac{1}{2}\right)
	&=\ip{\tau^\phi_{z-\frac{\iu}{2}}(A)^\ast J_\phi H^{\prime\frac{1}{2}} J_\phi\Phi}{J_\phi H^{\prime \frac{1}{2}} J_\phi\Phi}_\phi\\
	&=\ip{\tau^\phi_{z-\frac{\iu}{2}}(A)^\ast \Delta_\phi^{-\frac{1}{2}}J_\phi\Delta_\phi^{-\frac{1}{2}} H^{\prime \frac{1}{2}}\Phi}{\Delta_\phi^{-\frac{1}{2}}J_\phi \Delta_\phi^{-\frac{1}{2}} H^{\prime\frac{1}{2}} \Phi}_\phi\\
	&=\ip{\tau^\phi_{z-\frac{\iu}{2}}(A^\ast) \Delta_\phi^{-\frac{1}{2}} H^{\prime \frac{1}{2}} \Phi}{\Delta_\phi^{-\frac{1}{2}}H^{\prime \frac{1}{2}} \Phi}_\phi\\
	&=\ip{\Delta_\phi^{-\frac{1}{2}}\tau^\phi_{z-\frac{\iu}{2}}(A^\ast) \Delta_\phi^{-\frac{1}{2}} H^{\prime\frac{1}{2}} \Phi}{H^{\prime\frac{1}{2}} \Phi}_\phi\\	
	&=\ip{\tau^\phi_{z}(A^\ast) H^{\prime\frac{1}{2}} \Phi}{H^{\prime\frac{1}{2}} \Phi}_\phi\\	
	&=\ip{\tau^\phi_{z}(A^\ast) H^{\prime} \Phi}{ \Phi}_\phi\\	
	&=\psi\left(\tau^\phi_{z}(A)\right).	
	\end{aligned}
	\end{equation}
	
	Using the analyticity of the left-hand side of Equation \eqref{eq:calculationRN} and the constancy on the right-hand side on the line $\Im{z}=0$, remains no possibility but the constancy of the analytic extension for the strip $\strip{\frac{1}{2}}$.
	
	Finally, basically undoing the steps in Equation \eqref{eq:calculationRN}, we get, for every ${A\in \mathfrak{N}_\phi\cap \mathfrak{N}_\phi^\ast}$,
	$$\begin{aligned}
	\phi\left(H^\frac{1}{2} A H^\frac{1}{2}\right)&=\ip{A^\ast H^\prime\Phi}{\Phi}_\phi\\
	&=\ip{A^\ast J_\phi \Delta_\phi^{-\frac{1}{2}} H^\prime \Delta_\phi^{-\frac{1}{2}} J_\phi\Phi}{\Phi}_\phi\\
	&=\ip{A^\ast H\Phi}{\Phi}_\phi\\
	&=\phi(HA).
	\end{aligned}$$
	
	Hitherto we have proved that $\psi(A)=\phi(HA), \ \forall A \in \mathfrak{N}_\phi\cap \mathfrak{N}_\phi^\ast$, but this result can be extended for every $A\in\nalgebra_+$ throughout semifiniteness and normality.  
	
\end{proof}

\begin{corollary}
	\label{invariancecommute}
	Let $\phi$ and $\psi$ be two faithful normal semifinite weights on a von Neumann algebra $\nalgebra$, and $\{\tau^\phi_t\}_{t\in\mathbb{R}}$ and $\{\tau^\phi_t\}_{t\in\mathbb{R}}$ their automorphism groups, respectively. The following are equivalent:
	\begin{enumerate} [(i)]
		\item $\psi=\phi\circ\tau^\phi_t$
		\item $\{\tau^\phi_t\}_{t\in \mathbb{R}}$ and $\{\tau^\psi_t\}_{t\in\mathbb{R}}$ commute;
		\item $\phi=\phi\circ\tau^\psi_t$.
	\end{enumerate}  
\end{corollary}
\begin{proof}
	$(i)\Rightarrow (ii)$ By Corollary \ref{modularautomosphism}, $\displaystyle \tau^{\psi}_t=\tau^{\psi\circ\tau^\phi_t}_t={\tau^{\phi}}^{-1}_t\circ\tau^{\psi}_t\circ\tau^{\phi}_t\Rightarrow \tau^{\phi}_t\circ\tau^{\psi}_t=\tau^{\psi}_t\circ\tau^{\phi}_t$.
	
	$(ii)\Rightarrow (iii)$ If the automorphism groups commute, by Corollary \ref{modularautomosphism}, $\phi\circ\tau^\psi_t$ has $\{\tau^\phi_t\}_{t\in\mathbb{R}}$ as its automorphism group. In addition, it is obvious that $\phi\circ\tau^\psi_t$ is normal and semifinite.
	
	Then, also $\omega=\phi\circ\tau^\psi_t+\phi$ is normal and has $\{\tau^\phi_t\}_{t\in\mathbb{R}}$ as its automorphism group, let's prove it is semifinite.
	
	Since $\phi$ and $\phi\circ\tau^\psi_t$ are invariant under the action of $\{\tau^\phi_t\}_{t\in \mathbb{R}}$ so are $\mathfrak{N}_\phi$, $\mathfrak{N}_{\phi}^\ast$, $\mathfrak{N}_{\phi\circ\tau^\psi_t}$, and $\mathfrak{N}_{\phi\circ\tau^\psi_t}^\ast$. Thus, by item $(iv)$ in Proposition \ref{simplepropweights}, also $\mathfrak{M}_\psi$ and $\mathfrak{M}_{\phi\circ\tau^\psi_t}$ are invariant under the action of the automorphism group.
	
	We can use the very same proof of Proposition \ref{AnalDense} to obtain, for any $A\in \mathfrak{M}_\phi$, a sequence $(A_n)_n \subset \mathfrak{M}_\phi$ of analytic elements for $\{\tau^\phi_t\}_{t\in \mathbb{R}}$ such that $A_n\xrightarrow{WOT} A$ throughout Equation \eqref{formulaAE}. The same holds for any $A\in \mathfrak{M}_{\phi\circ\tau^\psi_t}$.
	
	Hence $\mathcal{M}_\phi\cap\mathcal{M}_\mathcal{A}$ and $\mathcal{M}_{\phi\circ\tau^\psi_t}\cap\mathcal{M}_\mathcal{A}$ are a WOT-dense in $\mathcal{M}_\psi$, which in its turn is WOT-dense in $\nalgebra$, because the multiplication is separately WOT-continuous.
	
	Finally, by Proposition \ref{twosidedmodule}, $\left(\mathcal{M}_\phi\cap\mathcal{M}_\mathcal{A}\right)\left(\mathcal{M}_{\phi\circ\tau^\psi_t}\cap\mathcal{M}_\mathcal{A}\right)\subset \mathcal{M}_\phi\cap\mathcal{M}_{\phi\circ\tau^\psi_t}$, and the set on the right-hand side is WOT-dense.
	
	Since $\omega$ is a faithful normal semifinite weight in $\nalgebra$, such that $\phi=\phi\circ\tau^\omega_t$ and $\phi\leq\omega$, by Theorem \ref{TPTRN}, there exists a unique operator $K\in \nalgebra$, invariant for $\{\tau^\omega_t=\tau^\phi_t\}_{t\in \mathbb{R}}$, such that $$
	\phi(A)=\omega(KA)=\phi\circ\tau^\psi_t(KA)+\phi(KA), \ \forall A\in \nalgebra_+.$$
	
	Since both $\phi$ and $\phi\circ\tau^\psi_t$ are faithful, $\{0,1\}\notin \sigma(K)$ and we can define the positive operator $H=\frac{K}{\mathbb{1}-K}$, which is affiliated to $\mathfrak{M}_{\tau\phi}$. Let $\left\{E^H_\lambda\right\}_{\lambda\in\mathbb{R}}$ be the spectral projections of $H$, we now that $\displaystyle\bigcup_{n\in\mathbb{N}}E^H_{(0,n)}\nalgebra_+ E^H_{(0,n)}$ is dense in $\nalgebra_+$ and
	\begin{equation}
	\label{eq:calculationX7}
	\phi(A)=\phi\circ\tau^\psi_t(HA), \  \forall A\in \bigcup_{n\in\mathbb{N}}E^H_{(0,n)}\nalgebra_+ E^H_{(0,n)}.
	\end{equation}
	
	By Lemma 1 in \cite{takesaki70} or Theorem 2.11 in \cite{Takesaki2003}, we know that the automorphism group for a weigh as in Equation \eqref{eq:calculationX7} is given by $$\begin{aligned}
	\tau^\phi_t(A)=\tau^\omega_t\left(H^{\iu t}AH^{-\iu t}\right)&=\tau^\phi_t(H^{\iu t}AH^{-\iu t})\\
	&\Rightarrow A=H^{\iu t}AH^{-\iu t}, \ \forall A\in \mathfrak{N}_\phi \\
	&\Rightarrow H\eta\mathcal{Z}(\nalgebra).
	\end{aligned}$$
	
	If $H\neq \mathbb{1}$, there exists projection $P\in\mathcal{F}_\phi \cap \mathfrak{M}_{\tau^\phi}$ such that $P\leq E^H_{(a,b)}$ for some $a,b\in\mathbb{R}$ with $1\neq(a,b)$. Then either $HP<P$ or $HP>P$, but this leads to an absurd since
	$$
	\phi(P)\neq\phi(HP)=\phi\circ\tau^\psi_t(P)=\phi\left(\tau^\psi_t(P)\right)=\phi(P).$$
	
	The conclusion is that $H=\mathbb{1}$ and then $\phi(A)=\phi\circ\tau^\psi_t(HA)=\phi\circ\tau^\psi_t(A)$.
	
	$(iii)\Rightarrow (i)$ is obvious just applying $(i)\Rightarrow(iii)$ for $\phi$ instead of $\psi$. 
\end{proof}

In order two generalise the previous theorem we will need the following lemmas.

\begin{lemma}
	\label{sumsemifinite}
	Let $\phi$ and $\psi$ be two faithful normal semifinite weights on a von Neumann algebra $\nalgebra$. Suppose $\psi$ is invariant under the modular automorphism group of $\phi$ $\{\tau^\phi_t\}_{t\in \mathbb{R}}$. Then $\phi+\psi$ is semifinite.
\end{lemma}
\begin{proof}
	
	Since $\psi$ is invariant under the action of $\{\tau^\phi_t\}_{t\in \mathbb{R}}$ and $\{\tau^\psi_t\}_{t\in \mathbb{R}}$, so is $\phi$, by Corollary \ref{invariancecommute}. Hence, as before $\mathfrak{M}_\phi$ and $\mathfrak{M}_\psi$ are invariant under the action of those automorphism groups.
	
	Since the automorphism groups commute Corollary \ref{invariancecommute}, we can use the very same proof of Proposition \ref{AnalDense} to obtain, for any $A\in \mathfrak{M}_\phi$, a sequence $(A_n)_n \subset \mathfrak{M}_\phi$ of analytic elements for both $\{\tau^\phi_t\}_{t\in \mathbb{R}}$ and $\{\tau^\psi_t\}_{t\in \mathbb{R}}$ such that $A_n\xrightarrow{WOT} A$ throughout the following expression
	$$A_n=\frac{n}{\pi}\int_\mathbb{R}\int_\mathbb{R}{e^{-n(t^2+s^2)}\tau^\psi_t\circ\tau^\phi_s(A) dt ds}.$$
	
	Hence $\mathcal{M}_\phi\cap\mathcal{M}_\mathcal{A}$ is a WOT-dense in $\mathcal{M}_\phi$, which in its turn is WOT-dense in $\nalgebra$. By the very same argument $\mathcal{M}_\psi\cap\mathcal{M}_\mathcal{A}$ is a WOT-dense in $\nalgebra$.
	
	Finally, $\left(\mathcal{M}_\phi\cap\mathcal{M}_\mathcal{A}\right)\left(\mathcal{M}_\psi\cap\mathcal{M}_\mathcal{A}\right)\subset \mathcal{M}_\phi\cap\mathcal{M}_\psi$ and the right-hand side is WOT-dense in $\nalgebra$, because the multiplication is separately WOT-continuous.
	
\end{proof}

\begin{theorem}[Pedersen-Takesaki-Radon-Nikodym]
	\label{TPTRN2}
	Let $\phi$ and $\psi$ be two normal semifinite weights on a von Neumann algebra $\nalgebra$. Suppose in addition that $\phi$ is faithful and $\psi$ is invariant under the modular automorphism group of $\phi$, $\{\tau^\phi_t\}_{t\in \mathbb{R}}$. Then, there exists a unique positive operator $H\eta\mathfrak{M}_{\tau^\phi}$ such that $\psi(A)=\phi(HA)$, $\forall A\in \nalgebra_+$.
\end{theorem}
\begin{proof}
	Let $s^\nalgebra(\psi)$ be the support projection for $\psi$. Notice that $\phi$ and $\psi$ are faithful normal semifinite weights for the von Neumann algebra $s^\nalgebra(\psi) \nalgebra s^\nalgebra(\psi)$ and $s^\nalgebra(\psi)$ is $\{\tau^\phi_t\}_{t\in \mathbb{R}}$-invariant. By Lemma \ref{sumsemifinite}, $\phi+\psi$ is a faithful normal semifinite weight on $s^\nalgebra(\psi) \nalgebra s^\nalgebra(\psi)$. Since $\phi\leq \phi+\psi$ and $\phi$ is $\left\{\tau^{\phi+\psi}_t\right\}$-invariant as a consequence of Corollary \ref{invariancecommute}, Theorem \ref{TPTRN} states that there exists a positive $\left\{\tau^{\phi+\psi}_t\right\}$-invariant operator $K\in s^\nalgebra(\psi) \nalgebra s^\nalgebra(\psi)$ with $0\leq K\leq \mathbb{1}$ such that
	$$\phi\left(s^\nalgebra(\psi)As^\nalgebra(\psi)\right)=\phi\left(Ks^\nalgebra(\psi)As^\nalgebra(\psi)\right)+\psi\left(Ks^\nalgebra(\psi)As^\nalgebra(\psi)\right), \ \forall A\in\mathfrak{N}_\phi.$$
	
	In addition, since $\phi$ is faithful on $s^\nalgebra(\psi) \nalgebra s^\nalgebra(\psi)$, $0\notin\sigma(K)$ and we can define $H=\frac{\mathbb{1}-K}{K}$. Let $\left\{E^H_\lambda\right\}_{\lambda\in\mathbb{R}_+}$ the spectral resolution of $H$ and define $H_n=HE_{(0,n)}=E_{(0,n)} H E_{(0,n)}$, then
	
	\begin{equation}
	\label{eq:calculationX8}
	\begin{aligned}
	\psi(E_{(0,n)}AE_{(0,n)})&=\psi(E_{(0,n)}A)\\
	&=\psi\left(s^\nalgebra(\psi)E_{(0,n)}As^\nalgebra(\psi)\right)\\
	&=\phi\left(H_n s^\nalgebra(\psi)E_{(0,n)}As^\nalgebra(\psi)\right)\\
	&=\phi\left(H_n E_{(0,n)} A\right)\\
	&=\phi\left(H_n^\frac{1}{2} AH_nH_n^\frac{1}{2}\right), \ \forall A\in\mathfrak{N}_\phi.
	\end{aligned}
	\end{equation}
	
	Using normality, it follows from Equation \eqref{eq:calculationX8} the desired invariance.
	
\end{proof}

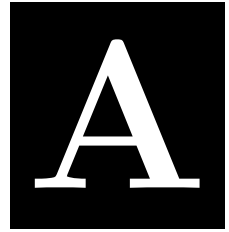
\renewcommand{\thechapter}{\arabic{chapter}}



\begin{appendices} 

\chapter{Geometric Hahn-Banach and Krein-Milman}
We will refer to the next result as Hahn-Banach Theorem, or more specifically, as Geometric Hahn-Banach Theorem but it was proved long after S. Banach's works or even H. Hahn did the generalization known nowadays as (Analytical) Hahn-Banach Theorem.

\begin{theorem}[de Mazur-Dieudonn\'e]
	\label{TMD}
	Let $V$ be a topological vector space, $M$ a subspace of $V$ and $A\subset X$ an open convex subset with $A\cap M=\emptyset$, then there exists a maximal closed subspace $H$ of $V$ disjoint of $A$ and containing $M$.
\end{theorem}

\begin{proof}
	First, suppose $V$ is a topological vector space over $\mathbb{R}$.
	
	Let $a\in A$, then, $A-a$ is an open and convex subset which contains the origin, thus the Minkowski functional $\rho_{A-a}$ is a continuous sublinear functional satisfying 
	\begin{equation}
	\label{mbola}
	A-a = \{x\in V| \rho_{A-a}(x)<1\} \textrm{ e } A = \{x\in V | \rho_{A-a}(x-a)<1\}
	\end{equation}
	
	So, since $M\cap A =\emptyset$, it follows that $\rho_{A-a}(x-a)\geq 1 \ \forall x \in M$.
	
	Define now $N=\langle M \cup \{a\} \rangle$ and let $\phi: N \to \mathbb{R}$ be given by $\phi(x-\lambda a)=\lambda$ which is clearly linear and:
	
	If $x\in M$ and $\lambda>0$ then $\displaystyle \phi(x-\lambda a)=\lambda \leq \lambda \rho_{A-a}\left(\frac{x}{\lambda} -a\right)=\rho_{A-a}(x-\lambda a)$.
	
	If $x\in M$ and $\lambda<0$ then $\phi(x-\lambda a)=\lambda <0 \leq \rho_{A-a}(x-\lambda a)$.
	
	This means $\phi$ is dominated by $\rho_{A-a}$ in $N$ and, using the Hahn-Banach Theorem, we obtain an extension $\tilde{\phi}$ of $\phi$ such that $\tilde{\phi}(x)\leq \rho_{A-a}(x) \ \forall x \in V$, in particular, $\tilde{\phi}$ is continuous, due to this and equation \ref{mbola}, we must have $\tilde{\phi}(x)<1 \ \forall x \in A-a$.
	
	Define the maximal subspace $H=\ker{(\tilde{\phi})}$, which is closed since $\tilde{\phi}$ is continuous, and contain $M=\ker{(\phi)}$. Moreover,
	$$x\in H \Rightarrow 0=\tilde{\phi}(x)=\tilde{\phi}(x-a) + \tilde{\phi}(a)=\tilde{\phi}(x-a)+\phi(a) \leq \rho_{A-a}(x-a)-1 $$
	It means, $\displaystyle \rho_{A-a}(x-a)\geq 1 \ \forall x\in H \Rightarrow H\cap A=\emptyset$ by \ref{mbola}.

	For the complex topological vector space case, $M_\mathbb{R} \subset V_\mathbb{R}$ is a subspace, and by the previous proof, there exists a maximal closed subspace $H$ of $V_\mathbb{R}$ disjoint of $A$ such that $M \subset H$. Identify the maximal closed subspace with the kernel of a continuous functional $\phi \in V_\mathbb{R}^*$ such that $H=\ker{(\phi)}$ and consider $\tilde{H} = H\cap \iu H$, which is also disjoint of $A$, and the functional $\tilde{\phi} \in V^*$ given by $\tilde{\phi}(x)=\phi(x)-\iu \phi(\iu x)$.
	$$\ker{(\tilde{\phi})}=\{x \in V \ | \ \phi(x)=0 \textrm{ e } \phi(\iu x)=0\}=H\cap iH$$ and thus $\tilde{H}$ is a maximal closed  subspace of $V$.
	$$M=\iu M \subset \iu H \textrm{ e } M \subset H \Rightarrow M \subset H\cap \iu H.$$
	
\end{proof}

\begin{corollary}
	\label{TMD2}
	Let $V$ be a topological vector space, $M$ an affine linear manifold of $V$ and $A\subset X$ an open convex subset disjoint of $M$, then there exists a closed hyperplane $H$ of $V$ disjoint of $A$ and containing $M$.
\end{corollary}

\begin{corollary}
	\label{CTMD}
	Let $V$ be a locally convex space, $M$ a closed affine linear manifold of $V$ and $K\subset X$ a compact convex subset which if disjoint of $M$, then there exists a closed hyperplane $H$ of $V$ which is disjoint of $K$ and contain $M$ .
	adon\end{corollary}
\begin{proof} Let $U$ be a neighbourhood of $0$ such that $(K+U)\cap M =\emptyset$. Since $V$ is locally convex, we can assume $U$ is a convex set, hence $K+U$ is an open convex (not empty) subset of $V$ which is disjoint of $M$. By Corollary \ref{TMD2}, there exists a hyperplane $H$ disjoint of $K$ and containing $M$.
	
\end{proof}

\begin{corollary}
	\label{C2TMD}
	Let $V$ be a locally convex space and $M \subset V$ a subspace. Then $x\in \overline{M}$ if, and only if, $x^*(x)=0$ for all $x^* \in V^*$ which vanish in $M$.
\end{corollary}
\begin{proof}
	
	($\Rightarrow$) Obvious.
	
	($\Leftarrow$) Of course $\overline{M}$ is a closed subspace of $V$, if $x \notin \overline{M}$ then we fall back in the conditions of Corollary \ref{CTMD} since $\{x\}$ is a compact convex set which does not intercept $M$. We conclude that there exists $x^* \in V^*$ such that $M\subset \ker{(x^*)}$ e $\{x\}\cap \ker{(x^*)}=\emptyset \Rightarrow x^*(x)\neq 0.$
	
\end{proof}

\begin{corollary}
	\label{C3TMD}
	Let $V$ be a locally convex space and $x\in V$, if $x^\ast(x)=0$ for all $x^* \in V^*$ then $x=0$.
\end{corollary}
\begin{proof}
	Of course, the set of functionals vanishing in the subspace $M=\{0\}$ is $V^\ast$. By Corollary \ref{C2TMD} we conclude that $x\in \overline{M}=\{0\}$.
	
\end{proof}

\begin{definition}[Face]
	
	Let $V$ a topological vector space and $C$ an convex subset, a non-empty closed and convex set $F\subset C$ is said to be an \emph{extremal set} or a \emph{face} of $C$ if given $x,y \in C$ and $\lambda \in (0,1)$ with $\lambda x +(1-\lambda)y \in F$ then $ x,y \in F$.
	
\end{definition}

That is, a face is a set such that if it contains any point in the interior of a straight segment, then it contains the whole segment.

\begin{definition}
	
	Let $V$ be an topological vector space and $C$ a convex subset. An extremal point in $C$ is a point $x\in C$ such that $\{x\}$ is a face of $C$.
	
	We denote by $\mathcal{E}(C)= \{x \in C  \ | \ x \textrm{ is an extremal point of C}\}$.
	
\end{definition}

\begin{proposition}
	Let $V$ be a topological vector space and $C$ a convex subset, the following conditions are equivalent:
	\begin{enumerate}[(i)]
		\item $x \in \mathcal{E}(C)$;
		\item $x=\lambda y +(1-\lambda) z$, with $y,z \in C$ and $\lambda \in (0,1) \Rightarrow x=y=z$; 
		\item $C\setminus\{x\}$ is convex;
	\end{enumerate}
\end{proposition}
\begin{proof}
	
	$(i)\Leftrightarrow (ii)$ It follows from definition.
	
	$(ii)\Leftrightarrow (iii)$ Let $y,z \in C\setminus\{x\}$ and $\lambda \in (0,1)$, then follows from convexity of $C$ that $\lambda y +(1-\lambda) z \in C$. On the other hand the extremicity of $x$ implies $\lambda y +(1-\lambda) z \neq x$ and then $C\setminus\{x\}$ is convex.
	
	Let $y,z \in C$  and $\lambda \in (0,1)$. Suppose by absurd that $y\neq x$, so we would have $z\neq x$ and hence $x=\lambda y +(1-\lambda) z \in C\setminus\{x\}$ since it is convex, a contradiction, so we must have $x=y=z$.
\end{proof}

\begin{proposition}
	\label{exex}
	Let $V$ be a locally convex space and $C\subset V$ a non-empty compact convex subset, then $\mathcal{E}(C)\neq\emptyset$.
\end{proposition}
\begin{proof}
	Denote by $\mathcal{F}$ the family of all faces of $C$ partially ordered by $\leq$, where $F_1\leq F_2$ if, and only if, $F_2\subset F_1$. Of course such family is non-empty since $C\in\mathcal{F}$.
	
	Let $\mathcal{F}_0\subset\mathcal{F}$ be a chain. Since $\mathcal{F}_0$ is a totally ordered set and its elements are closed, $\mathcal{F}_0$ has the finite intersection property and it follows from compactness of $C$ that $\displaystyle \bigcap_{F\in \mathcal{F}_0}F \neq \emptyset$. It is easy to see that $\displaystyle \bigcap_{F\in \mathcal{F}_0}F$ is a face of $C$, since intersections preserve the closed convex and extreme properties of a set, furthermore, it is clearly an upper bound of $\mathcal{F}_0$. It follows then by Zorn's lemma that there exists $\tilde{F} \subset \mathcal{F}$ maximal.
	
	Let us show $\tilde{F}$ is a unitary set. In order to do that, suppose it is not true, that is, take $x,y \in \tilde{F}$ with $x\neq y$. Consider then the compact and convex set $D=\{x\}$ and the linear affine manifold $M=\{0\}+y$, by Theorem \ref{TMD2} there exists a hyperplane $H$ which contains $M$ and is disjoint of $D$.
	
	Let $f \in V^*$ and $c \in \mathbb{K}$ such that $H=f^{-1}(\{c\})$ and define 
	$$ \tilde{F}_0=\left\{x\in \tilde{F} \ \middle | \ f(x)=\inf_{y \in \tilde{F}}{f(x)}\right\}.$$
	
	Lets show that $\tilde{F}_0$ is a face of $\tilde{F}$ and consequently a face of $C$, in fact, suppose $x=\lambda y + (1-\lambda) z$ with $x\in \tilde{F}_0$, $y,z \in \tilde{F}$ and $\lambda \in (0,1)$, then
	$$\inf_{y \in \tilde{F}}{f(x)}=f(\lambda y + (1-\lambda) z)=\lambda f(y) + (1-\lambda) f(z) \leq \inf_{y \in \tilde{F}}{f(x)} \Rightarrow f(x)=f(y)=f(z)$$
	It follows from the definition that $y,z \in \tilde{F}_0$ and thus it is a face. Note now that we cannot have $x$ and $y$ simultaneously as elements of $\tilde{F}_0$, so $\tilde{F} \leq \tilde{F}_0$ and this contradicts the maximality of $\tilde{F}$. Therefore $\tilde{F}$ is unitary and this guarantees the existence of maximal points of $C$.
	
\end{proof}

\begin{theorem}[Krein-Milman]
	\label{TKM}
	Let $V$ be a locally convex space and $C$ a compact convex subset, then $C=\cchull{\mathcal{E}(C)}$.
\end{theorem}
\begin{proof}
	By Proposition \ref{exex}, $\mathcal{E}(C)\neq \emptyset$. Suppose that $C\setminus \cchull{\mathcal{E}(C)}\neq \emptyset$ and take $x\in C\setminus \cchull{\mathcal{E}(C)}$, note that $\cchull{\mathcal{E}(C)} \subset C$ is compact since it is a closed set contained in a compact one. By theorem \ref{TMD2} there exists a closed hyperplane $H$ containing $\{x\}$ and disjoint of $\cchull{\mathcal{E}(C)}$.
	
	Let $f \in V^*$ and $c \in \mathbb{K}$ such that $H=f^{-1}(\{c\})$, we can assume without loss of generality that $f(x)=c < f(y) \ \forall y \in \cchull{\mathcal{E}(C)}$, and take $\displaystyle F=\left\{x\in C \ \middle | \ f(x)=\inf_{y \in C}{f(x)}\right\}$. We have already seem in the proof of Proposition \ref{exex} that $F$ is a proper face of $C$ and that $\mathcal{E}(F)\neq \emptyset$, furthermore, if $y\in \mathcal{E}(F)$, $\displaystyle f(y)=\inf_{z \in C}{f(z)} \leq f(x) = c < f(w) \ \forall w \in \cchull{\mathcal{E}(C)}$, hence $y \notin \cchull{\mathcal{E}(C)}$. On the other hand we must have $\mathcal{E}(F)\subset \mathcal{E}(C)$. This leads us to a contradiction, so we conclude that $C=\cchull{\mathcal{E}(C)}$.
	
\end{proof}

\begin{lemma}
	\label{extremealidentity}
	Let $\calgebra$ be a $C^\ast$-algebra and $\mathcal{S}$ be its unit sphere, then $\mathcal{S}$ has an extremal point if, and only if, $\calgebra$ has an identity.
\end{lemma}
\begin{proof}
	$(\Rightarrow)$ If $\calgebra$ has an identity $\mathbb{1}$, we can write it as $\mathbb{1}=\frac{A+B}{2}$ with $A,B \in \mathcal{S}$. It follows that $\mathbb{1}= \frac{\tilde{A}+\tilde{B}}{2}$ with $\tilde{A}=\frac{A+A^\ast}{2}$ and $\tilde{B}=\frac{B+B^\ast}{2}$. Now, since $\tilde{A}=2\mathbb{1}-\tilde{B}$, they are self-adjoint elements (they commute) and, by the Spectral Theorem, $\tilde{A},\tilde{B}\geq \mathbb{1}$. But $\tilde{A},\tilde{B} \in \mathcal{S}$, hence $\tilde{A}\leq \mathbb{1}$ and $\mathbb{1}\leq\tilde{B}$, thus we conclude that $\tilde{A}=\tilde{B}=\mathbb{1}$.
	
	Returning to the definition of the operators, $A=2\mathbb{1}-A^\ast$ and it follows that $A$ is a normal operator such that $2A=AA^\ast+A^\ast A$, hence positive. By the above argument $A=\mathbb{1}$ and the using the analogous argument for $B$, it follows that $\mathbb{1}$ is an extremal point.
	
	$(\Leftarrow)$ 
	Suppose now $A \in \mathcal{E}(\mathcal{S})$. Of course $\sigma(A^\ast A),\sigma(AA^\ast)\subset {0,1}$, otherwise it is easy to construct (using the Spectral Theorem) a positive operator $B\subset \calgebra$ such that $\|B\| \leq 1$, $\|A \pm B\| =1$, in particular, $A^\ast A$ and $AA^\ast$ are projections.
	
	
	Now, let $B\in \{C-C A^\ast A-A^\ast A C A^\ast A+A^\ast A C A^\ast A  \ | \ C \in \calgebra\}$ such that $\|B\|\leq1$. A straight forward calculation using that $A^\ast A$ is a projection shows that $B^\ast A A^\ast B=0$, thus $\| A^\ast B\|=\|B^\ast A(B^\ast A)^\ast\|^{\frac{1}{2}}=0 \Rightarrow B^\ast A=A^\ast B=0$ and $A^\ast A B^\ast B=0$.
	So, we must have 
	\begin{equation}\label{eq2}\begin{aligned}
	\|A\pm B\|&
	=\|(A^\ast\pm B^\ast)(A\pm B)\|^{\frac{1}{2}}\\
	&=\|A^\ast A\pm(A^\ast B+B^\ast A)-B^\ast B\|^{\frac{1}{2}}\\
	&=\|A^\ast A+B^\ast B\|^{\frac{1}{2}}=\max\{\|A\|,\|B\|\}\\
	&\leq 1.\end{aligned} \end{equation}
	
	From equation \ref{eq2} we conclude, since $A$ is an extremal point,
	$$
	A=\frac{1}{2}\frac{A+B}{\|A+B\|}+\frac{1}{2}\frac{A-B}{\|A-B\|} \Rightarrow B=0 \Rightarrow$$
	$$\{C-C A^\ast A-A A^\ast C +A A^\ast C A^\ast A  \ | \ C \in \calgebra\}=\{0\}
	$$
	
	Now, define $h=A^\ast A+AA^\ast$, and suppose it does not have an inverse, that means, via the identification in \ref{TGN} and Theorem \ref{TI}, there exists a positive operator $B \in \calgebra$ with $\|B\|=1$ and $hB=0$. But then $$\|AB\|=\|B A^\ast\|=\|B A^\ast A B\|^{\frac{1}{2}}\leq \|B h B\|=0.$$
	
	Doing the analogous estimation to $\|BA\|$ we conclude
	$$\| B-B A^\ast A-A A^\ast B +A A^\ast B A^\ast A\|=\|B\|= 1.$$
	
	Since it is a contradiction we must have that $hh^{-1}$ is an identity.
\end{proof}

The next result is a original proof of a well known result in the theory of $C^\ast$-algebras.

\begin{theorem}[Segal]
	\label{ExisAppId}
	Every $C^\ast$-algebra $\calgebra$ contains a positive
	approximate identity.
\end{theorem}
\begin{proof}
	
	First we recall Theorem \ref{TGN}, to reduce to the case of a subspace of $B(\hilbert)$.
	
	Let $\mathcal{P}=\{A\in\calgebra \ | \ A\geq0\}$. Then $\mathcal{P}$ is a convex pointed cone, since by the Banach-Alaoglu Theorem the unit ball is weak-operator compact, $K=\overline{\mathcal{P}\cap B_{B(\hilbert)}}^{WOT}=\overline{\mathcal{P}}^{WOT}\cap\overline{\mathcal{S}_{B(\hilbert)}}^{WOT}$ must be so, thus $K=\overline{conv}^{WOT}{\left(\mathcal{E}(K)\right)}$ due to Theorem \ref{TKM}. On the other hand, according to Lemma \ref{extremealidentity} there must exists an identity in the weak-operator closure of $\mathcal{S}_{B(\hilbert)}$. Let $(I_\alpha)_\alpha \subset \mathcal{S}_{B(\hilbert)}$ be a net, which is convergent to this identity denoted by $\mathbb{1}$. This means, for all $x, y\in \hilbert$ we have $\ip{x-I_\alpha x}{y}\rightarrow 0$.
	
	For each index $\alpha$ there exists a unique positive operator $\sqrt{\mathbb{1}-I_\alpha}$ such that $\sqrt{I_\alpha-\mathbb{1}}^2=\mathbb{1}-I_\alpha$. It follows that
	$$\begin{aligned}
	\|x-I_\alpha x\|^4
	&=\ip{(\mathbb{1}-I_\alpha)^{\frac{1}{2}}x}{(\mathbb{1}-I_\alpha)^{\frac{3}{2}}x}\\
	&\leq \|(\mathbb{1}-I_\alpha)^{\frac{1}{2}}x\| \|(\mathbb{1}-I_\alpha)^{\frac{3}{2}}x\|\\
	&\leq \|I_\alpha\|^{3} \|x\|^2 \ip{x-I_\alpha x}{x}\\
	&\leq \|x\|^2 \ip{x-I_\alpha x}{x}\\
	\end{aligned}$$
	and from this we conclude, for each fixed $x\in \hilbert$, $\|I_\alpha x -x\|\rightarrow 0$, $I_\alpha \in \mathcal{P}\cap \mathcal{S}_\calgebra$.
\end{proof}

\chapter{Representations and Spectral Analysis}

Spectral theory is well known for normal operators in $B(\hilbert)$, the space of bounded operators on a Hilbert space $\hilbert$, it allows us to construct a functional calculus for operators. It is also known that spectral theory and functional calculus can be extended to $C^\ast$-algebras. An easy way to import all results of spectral theory of bounded operators are going to be shown next.

\begin{definition}[Resolvent and Spectrum]
	Let $\calgebra$ be a $C^\ast$-algebra, where we adjoin an identity if none is provided. Let $A\in \calgebra$.
	\begin{enumerate}[(i)]
		\item $\rho(A)=\{\lambda \in \mathbb{C} \ |\ (\lambda\mathbbm{1}-A)\textrm{ has an inverse in } \calgebra\}$ is called the resolvent\index{resolvent} of $A$.
		
		\item $\sigma(A)=\mathbb{C}\setminus\rho(A)$ is called the spectrum\index{spectrum} of $A$.
		
		\item $r(A)=\sup\{|\lambda| \ |\ \lambda \in\sigma(A)\}$ is called the spectral radius \index{spectral radius} of $A$.
	\end{enumerate}
\end{definition}

Note that this is the usual definition of spectrum in $B(\hilbert)$ and that the spectral radius is a positive finite number due to the next lemma.

\begin{remark} \
	We do not specify in which algebra the spectrum is taken, when necessary we will write $\rho_\algebra(A)$ in order to fix the algebra, but it is an interesting fact (it will be shown later) that the spectrum does not depend on the $C^\ast$-algebra.
	
\end{remark}

\begin{lemma}
	\label{bspec}
	If $A\in \calgebra$ and $|\lambda|>\|A\|$, then $\lambda \in \rho(A)$, in other words, $\lambda\in \sigma(A) \Rightarrow |\lambda| \leq \|A\|$.
\end{lemma}
\begin{proof}
	Define $\displaystyle B_m=\sum_{n=0}^{m}{\lambda^{-(n+1)}A^n}$. This is an absolutely convergent series and, in special, $\calgebra$ is a Banach Space, consequently there exists $\displaystyle B=\lim_{m\rightarrow \infty}{B_m}$.
	
	In order to show that $B$ is the inverse of $\lambda\mathbbm{1}-A$, note that
	$$\begin{aligned}
	(\lambda\mathbbm{1}-A)B_m	&=(\lambda\mathbbm{1}-A)\sum_{n=0}^{m}{\lambda^{-(n+1)}A^n} \\
	&=\sum_{n=0}^{m}{\lambda^{-n}A^n}-\sum_{n=1}^{m+1}{\lambda^{-n}A^n}\\
	&=\mathbbm{1}-\lambda^{-(m+1)}A^{m+1}\\
	\end{aligned}$$
	since $B_m\rightarrow B$, the left-hand side converges to $(\lambda\mathbbm{1}-A)B$ while by $\lambda^{-(m+1)}A^{m+1}\rightarrow 0$ the right-hand side goes to $\mathbbm{1}$.
	
	The proof for $B(\lambda\mathbbm{1}-A)$ follows by the same argument.
\end{proof}

\begin{theorem}[First Theorem of Isomorphism for Banach Algebras]
	{\label{TI}}
	Let $\algebra_1,\algebra_2$ be Banach algebras and $\Phi:\algebra_1 \to \algebra_2$ a bounded $\ast$-homomorphism, then there exists a unique $\ast$-isomorphism $\tilde{\Phi}: \algebra_1/\ker{(\Phi)} \to \Ran{\Phi}$ such that the following diagram commutes.
	
	\begin{center}
		\begin{minipage}{7 cm}
			\begin{displaymath}
			\xymatrix{ \algebra_1 \ar[dr]_{\eta} \ar[rr]^\Phi & 			 											&	\Ran(\Phi) \subset \algebra_2						\\
				& \algebra_1/\ker{(\Phi)}	\ar[ur]_{\tilde{\Phi}} 									}
			\end{displaymath}
		\end{minipage}
		\hfil\hspace{-4cm}
		.
	\end{center}
	
	Furthermore, $\|\Phi\|=\|\tilde{\Phi}\|$.
\end{theorem}
\begin{proof}
	Define $\tilde{\Phi}: \algebra_1/\ker{(\Phi)} \rightarrow \Ran(\Phi)$ such that $\tilde{\Phi}([A])=\tilde{\Phi}(\eta(A))=\Phi(A)$.
	
	Of course we can check that this function is well defined, since we used a representing element of the equivalent class to define the function. If $[A_1]=[A_2] \in \algebra_1/\ker{(\Phi)}$ we have $A_1 = A_2 + K$ with $K \in \ker{(\Phi)}$ and then 
	$$\tilde{\Phi}([A_1])=\Phi(A_1)=\Phi(A_1)+ \Phi(K)= \Phi(A_1 + K) = \Phi(A_2) = \tilde{\Phi}([A_2]).$$
	
	$\tilde{\Phi}$ is again a $\ast$-homomorphism because
	$$A_1, A_2 \in \algebra_1
	\begin{cases} &\Rightarrow \tilde{\Phi}([A_1 + A_2])= \Phi(A_1+A_2)=\Phi(A_1)+\Phi(A_2) =\tilde{\Phi}([A_1])+\tilde{\Phi}([A_2]).\\
	&\Rightarrow \tilde{\Phi}([A_1 A_2])= \Phi(A_1A_2)=\Phi(A_1)\Phi(A_2)=\tilde{\Phi}([A_1])\tilde{\Phi}([A_2]).\\
	&\Rightarrow \tilde{\Phi}([A_1^\ast])= \Phi(A_1^\ast)=\Phi(A_1)^\ast =\tilde{\Phi}([A_1])^\ast.\\
	\end{cases}$$
	
	In order to show that $\tilde{\Phi}\in \mathcal{B}(\algebra_1/\ker{(\Phi)},\Ran(\Phi))$, it is enough to note that
	$$\begin{aligned}
	\|\tilde{\Phi}([A])\| 
	&= \inf_{\tilde{A}\in \ker{(\Phi)}}{\|\Phi(A+\tilde{A})\|}\\
	& \leq \inf_{\tilde{A}\in \ker{(\Phi)}}{\|\Phi\|\ \|A + \tilde{A}\|}\\
	& = \|\Phi\| \inf_{\tilde{A}\in \ker{(\Phi)}}{\|A + \tilde{A}\|}\\
	&=\|\Phi\| \ \|\ [A]\ \|.\\
	\end{aligned}$$
	hence $\|\tilde{\Phi}\| \leq \|\Phi\|$. In addition, 
	$$\|\Phi(A)\| = \|\tilde{\Phi}\circ \eta (A)\| = \|\tilde{\Phi}([A])\| \leq \|\tilde{\Phi}\| \ \| \ [A] \ \| \leq \|\tilde{\Phi}\| \ \|A\|$$
	and from the two inequalities it follows that $\|\Phi\| = \|\tilde{\Phi}\|$.
	
	Finally, $\tilde{\Phi}$ is bijective because $\tilde{\Phi}([A])=\Phi(A)= 0 \Leftrightarrow A\in \ker{(\Phi)} \Leftrightarrow [A]=[0] $ and $y\in \Ran(\Phi) \Leftrightarrow y=\Phi(A)$ for some $A\in \algebra_1$ thus, $\tilde{\Phi}([A])=\Phi(A)=y$.
	
	We define $\tilde{\Phi}$ in such a way that $\Phi=\tilde{\Phi} \circ \eta$ and, in addition, if $\tilde{\Psi}$ is another continuous operator satisfying this identity we have $(\tilde{\Phi}-\tilde{\Psi})\eta(A) = 0 \ \forall A\in \algebra_1$ and it follows that $\tilde{\Phi}=\tilde{\Psi}$.
\end{proof}

Although we are interested only in the case of Banach algebras, it is interesting to note that multiplication and involution play no role at the definition of the operator $\tilde{\Phi}$, that is, the very same definition works to prove this kind of theorem for groups, for example. 

\begin{definition}[Representation]
	A representation \index{representation} $\pi$ of a $C^\ast$-algebra $\calgebra$ is a $\ast$-homomorphism from $\calgebra$ into $B(\hilbert)$ for some Hilbert space $\hilbert$. The representation is said faithful \index{representation! faithfull} if $\pi$ is injective.
\end{definition}

\begin{theorem}
	\label{nonexpensive}
	Let $\calgebra$ be a $C^\ast$-algebra and $\pi$ representation, then $\|\pi(A)\|\leq \|A\|$.
\end{theorem}
\begin{proof}
	From the Theorem \ref{TI}, there exists a $\ast$-isomorphism $\tilde{\pi}: \calgebra/\ker{(\pi)} \to \Ran{\pi}$.
	
	Since $\tilde{\pi}$ is a $\ast$-isomorphism, $\tilde{\pi}([\mathbbm{1}])=\pi(\mathbbm{1})$ is in the image of $\pi$. Furthermore, $\lambda[\mathbbm{1}]-[A]$ has an inverse if, and only if, $\lambda\tilde{\pi}([\mathbbm{1}])-\tilde{\pi}([A])$ is invertible too. So, we have the identification $\sigma\left([A]\right)=\sigma\left(\tilde{\pi}([A])\right)$, in particular, from Lemma \ref{bspec} it follows that $r(A)\leq \|A\|$.
	
	Now, we can use known results of (classical) spectral analysis for a normal element $\tilde{\pi}([A][A]^\ast)$, from which we conclude
	$$\left\|\pi(A)\right\|^2=\left\|\tilde{\pi}([A])\right\|^2=\left\|\tilde{\pi}([A][A]^\ast)\right\|=r\left(\tilde{\pi}([A][A]^\ast)\right)\leq \|AA^*\|=\|A\|^2.$$
	
\end{proof}


\begin{definition}[Pure state]
	A state in a $C^\ast$-algebra $\calgebra$ is said a pure state \index{state! pure} if it is an extremal point of $\mathcal{S}_\calgebra=\{\omega \in \calgebra^\ast \ | \ \|\omega \|=1\}$. When $\omega$ is not pure it is called a mixed\index{state! mixed} state.
\end{definition}

\begin{notation}
	We denote by $\mathcal{E}(X)$ the set of extremal points \index{extremal points} of $X$.
	
	(The definition of extremal point can be found in Appendix A).
\end{notation}

For now, the existence of pure states is not clear, but it will shown soon that they exist in sufficient number such that its closed convex hull coincides with the weak$^\ast$-closure of $B_{\calgebra^\prime}$.

\begin{proposition}
	\label{nullpurestates}
	Let $A\in\calgebra$. If $\omega(A)=0$ for each pure state $\omega$, then $A=0$.
\end{proposition}
\begin{proof}
	By Theorem \ref {TKM}, $B_{\calgebra^\prime}=\cchull{\mathcal{E}(\mathcal{S}_\calgebra)}$, where the closure is taken in the weak-$\ast$ topology. Then by continuity of the functionals in $\mathcal{S}_\calgebra$ we must have $\rho(A)=0 \ \forall \rho \in B_{\calgebra^\prime}$ and using Corollary \ref{C3TMD}, this implies $A=0$.
	
\end{proof}

\begin{lemma}
	\label{existencenorm}
	Let $\calgebra$ be a $C^\ast$-algebra and $A \in \calgebra$, there exists a pure state $\omega$ such that $\omega(A^\ast A)=\|A\|^2$.
\end{lemma}
\begin{proof}
	If $\calgebra$ has no unity, add it. Consider the subalgebra 
	$$\mathfrak{A}_0=\{\alpha\mathbbm{1}+\beta A^\ast A \ | \ \alpha, \beta \in \mathbb{K}\}.$$
	
	Define the linear functional $\tilde{\omega}: \mathfrak{A}_0 \to \mathbb{K}$ by $\tilde{\omega}(\alpha \mathbbm{1}+\beta A^\ast A)= \left|\alpha +\beta \|A\|^2\right|$.
	
	Note that
	$$\begin{aligned}
	\tilde{\omega}(\alpha \mathbbm{1}+\beta A^\ast A) &=\left|\alpha +\beta \|A\|^2\right| \\
	&\leq \sup_{\lambda \in \sigma(A^\ast A)}{|\alpha + \beta\lambda|}\\
	& r(\alpha \mathbbm{1}+\beta A^\ast A)\\
	&=\|\alpha \mathbbm{1}+\beta A^\ast A\|.
	\end{aligned}$$
	Hence, since $\tilde{\omega}(\mathbbm{1})=1$, $\tilde{\omega}$ is a state on $\mathfrak{A}_0$. Now, by Hahn-Banach Theorem it has a norm preserving extension to $\calgebra$.
	
	This warrants that the closed and convex set 
	$$\mathfrak{F}=\left\{\omega\in \mathcal{S} \ \middle| \ \omega(A^\ast A)=\|A\|^2\right\}\neq \emptyset,$$
	thus it has a extreme point by Theorem \ref{TKM}.
	
	Let $\omega$ be such an extreme point, and suppose $\omega=\lambda \omega_1+(1-\lambda)\omega_2$ for $\omega_1, \omega_2 \in \mathcal{S}$ and $0<\lambda<1$, then
	$$\begin{aligned}
	\|A\|^2
	&=\omega(A^\ast A)\\
	&=\lambda \omega_1(A^\ast A)+(1-\lambda)\omega_2(A^\ast A)\\
	&\leq \lambda \|A\|^2+(1-\lambda)\|A\|^2\\
	&=\|A\|^2\\
	&\Rightarrow \omega_1(A^\ast A)=\omega_2(A^\ast A)=\|A\|^2\\
	&\Rightarrow \omega_1,\omega_2 \in \mathfrak{F}\\
	&\Rightarrow \omega=\omega_1=\omega_2.
	\end{aligned} $$
	
	Hence, $\omega$ is an extremal point in $\mathcal{S}$ as well and satisfies $\omega(A^\ast A)=\|A\|^2$. 
	
\end{proof}

\begin{proposition}[GNS-Representation]
	\label{GNS}
	Let $\calgebra$ be a $C^\ast$-algebra and $\omega$ a state. Then there exists a Hilbert space $\hilbert_\omega$ and a representation $\pi_\omega$ of $\calgebra$ in $B(\hilbert_\omega)$. This representation also admits a cyclic vector $\xi$ for which
	$$\omega(A)=\ip{\pi_\omega(A)\xi}{\xi}.$$
\end{proposition}
\begin{proof}
	Define the closed left ideal (two-sided, due to $(i)$ in Proposition \ref{cauchyschwarz}, as it will be became clear bellow) $N_\omega=\left\{A \in \calgebra \ \middle| \ \omega(A^\ast A)=0 \right\}$ and $\hilbert_\omega=\calgebra/N_\omega$ provided with the inner product $\ip{[A]}{[B]}=\omega(A^\ast B)$.
	
	First, this inner product is well defined because, if $N_1, N_2\in N_\omega$,
	\begin{equation}
	\label{eq:calculationquotient}
	\begin{aligned}
	\ip{[A+N_1]}{[B+N_2]}
	& = \omega\left((A+N_1)^\ast (B+N_2)\right)\\
	&= \omega(A^\ast B)+\omega(N_1^\ast B)+\omega(A^\ast N_2)+\omega(N_1^\ast N_2)\\
	&= \omega(A^\ast B)+\overline{\omega(B^\ast N_1)}+\omega(A^\ast N_2)+\omega(N_1^\ast N_2)\\
	&=\omega(A^\ast B)\\
	&=\ip{[A]}{[B]}\\
	\end{aligned}
	\end{equation}
	where we used Proposition \ref{cauchyschwarz} $(i)$ and that $N_\omega$ is a left ideal.
	
	Positivity and sesquilinearity follow trivially from the positivity and linearity of $\omega$ and from the anti-linearity of $\ast$. It still remains to prove that $\ip{[A]}{[A]}=0 \Rightarrow [A]=0$, but this follows from the definition of quotient.
	
	Now, let us define the representation. Define the left representation by $$\pi_\omega(A)\left([B]\right)=[A B]=[A][B].$$
	
	Of course $\pi_\omega(A)$ is linear, $\pi_\omega(A B)=\pi_\omega(A)\pi_\omega(B)$ and  $\pi_\omega(A^\ast)=\pi_\omega(A)^\ast$, thus $\pi_\omega$ is a $\ast$-homomorphism. By definition of the quotient norm $\left\|\pi_\omega(A)\left([B]\right)\right\|=\left\|[A][B]\right\|\leq \left\|[A]\|\|[B]\right\|\leq \|A\| \|B\|$ and that means $\pi_\omega(A) \in B(\hilbert_\omega)$.
	
	It remains just to prove the existence of a cyclic vector. Let $(e_\lambda)_{\lambda\in\Lambda}$ an increasing approximate identity (\ref{ExisAppId}), then the equivalent classes $([e_\lambda])_{\lambda\in\Lambda}$ form a bounded norm increasing sequence of vectors, thus convergent to some $\xi\in \hilbert_\omega$.
	
	It follows from the definition that this is a cyclic vector and
	$$\omega(A)=\sup_{\lambda \in \Lambda} \omega(e_\lambda A e_\lambda)=\sup_{\lambda \in \Lambda}\ip{\pi_\omega(A)[e_\lambda]}{[e_\lambda]}=\ip{\pi_\omega(A)\xi}{\xi}.$$
\end{proof}

\begin{remark}
	We emphasize that $N_\omega$ is a two-sided ideal, but we only use it is a left ideal. This is because we could define a right representation in a analogous way which is closely related to the modular operator.
\end{remark}

Note that the representation obtained in the previous result is not faithful. In fact, it is faithful if $N_\omega=\{0\}$. This points to the following definition:

\begin{definition}
	A state $\omega$ is said to be faithful \index{state! faithfull} if $\omega(A^\ast A)=0 \Leftrightarrow A=0$.
\end{definition}
\begin{remark}
	The representation of Proposition \ref{GNS} is faithful in $\calgebra/N_\omega$.
\end{remark}

\begin{theorem}[Gelfand-Naimark]
	\label{TGN}
	Every $C^\ast$-algebra $\calgebra$ admits a faithful (isometric) representation.
\end{theorem}
\begin{proof}
	Let $\mathcal{E}$ be the set of all pure states on $\calgebra$ and let $\hilbert_\omega$ and $\pi_\omega$ be the Hilbert space and the corresponding representation, respectively,  obtained in Theorem \ref{GNS}. Define
	$$
	\hilbert =\bigoplus_{\omega\in\mathcal{E}}{\hilbert_\omega}, \quad
	\pi=\bigoplus_{\omega\in\mathcal{E}}{\pi_\omega}.
	$$
	
	By Proposition \ref{nullpurestates} $\pi(A)=0 \Leftrightarrow A=0$, thus $\pi$ is a $\ast$-isomorphism. Now it follows from Theorem \ref{nonexpensive} that $\|\pi(A)\|\leq \|A\|$. On the other hand, Lemma \ref{existencenorm} warrants the existence of a pure state $\omega$ such that $\omega(A^\ast A)=\|A\|^2$. Hence $\|A\|^2\leq \|\pi(A)\|^2$ and the equality follows.
	
\end{proof}

This theorem is the first indicative of the necessity of weights in von Neumann algebras, because although the direct sum works well for defining the new Hilbert space as well as the representation, this representation is not related to a state because the sum of the pure states may diverge, but for weights however this is not a problem.

Some of the next results can be found in \cite{Takesaki2003}.

\begin{proposition}[GNS-Representation for Weights]
	\label{GNSweight}
	Let $\calgebra$ be a $C^\ast$-algebra and $\phi$ a weight, then there exists a Hilbert space $\hilbert_\phi$ and a representation $\pi_\phi$ of $\calgebra$ in $B(\hilbert_\phi)$.
\end{proposition}
\begin{proof}
	The proof follows exactly the same steps of Proposition \ref{GNS}, just defining the Hilbert space $\hilbert_\phi$ as the completion of the pre-Hilbert space $\mathfrak{N}_\phi/N_\phi$ and the representation
	$$\begin{aligned}
	\pi_\phi: \ & \calgebra	&\to	\ & B\left(\hilbert_\phi\right)	& \quad & \quad & \pi_\phi(A): \	& \hilbert_\phi  &\to \ & \ \hilbert_\phi \	   \\
	\quad	 & A			&\mapsto \ & \ \pi_\phi(A) 						& \quad & \quad & \	& B			\ &\mapsto \ & [AB]	\\
	\end{aligned}$$
	because Proposition \ref{simplepropweights} warrants the required properties. 
	
\end{proof}


\begin{notation}
	\label{GNSnotation}
	Throughout this work we will denote by $\pi_\phi$ any representation (in particular, the GNS-representaion related to the weight $\phi$) such that $\pi_\phi(A):\hilbert_\phi \to \hilbert_\phi$ and $\ip{\pi_\phi(A)}{\pi_\phi(B)}_\phi=\phi(B^\ast A)$ and $\pi_\phi(A)\pi_\phi(B)=\pi_\phi(A B)$.
\end{notation}

\begin{lemma}
	\label{lemmax1}
	Let $\calgebra$ be a $C^\ast$-algebra, $A\in \calgebra$ a normal operator, $K\subset \mathbb{R}$ a compact set with $-\|A\|,\|A\|\in K$, $|k|\leq\|A\| \ \forall k\in K$ and $p:K \to \mathbb{R}_+$ a polynomial. Then $\|p(A)\|\leq p(\|A\|)$. 
\end{lemma}

\begin{proof}
	Proceeding by induction on the degree of $p$. If $degree(p)=0$, there is nothing to prove. Now suppose the statement is true for any positive polynomial on $K$ of degree less then $n\in \mathbb{N}$ and take $p$ a positive polynomial with $degree(p)=n+1$ also on $K$. Decompose 
	$$\begin{aligned}
	p(x)	&=xq(x)+r, \ degree(q)= n\\
	&=(xq(x)-\min{xq(x)})+\left(r+\min{xq(x)}\right), \ degree(q)= n
	\end{aligned}$$
	
	It is important to notice that $xq(x)-\min{xq(x)}\geq0$ and $\min{p(x)}=r+\min{xq(x)}\geq0$, and just using the $C^\ast$-condition and the induction hypothesis we get $$\begin{aligned}
	\|p(A)\| &\leq\left\|Aq(A)-\min{xq(x)}\mathbbm{1}\right\|+\left\|\left(r+\min{xq(x)}\right)\mathbbm{1}\right\|\\
	&\leq\left(\|A\|q(\|A\|)-\min{xq(x)}\right)+\left(r+\min{xq(x)}\right)\\
	&=p(\|A\|).
	\end{aligned}$$
	
\end{proof}

\begin{theorem}[Functional Calculus]
	\label{funccalculus}
	Let $\calgebra$ be a $C^\ast$-algebra and $A\in \calgebra$ a self-adjoint operator. There exists a unique isometric $\ast$-isomorphism $$\Psi:\mathcal{C}\left(\sigma(A)\right)\to C^\ast\left(\{\mathbbm{1},A\}\right)$$
	such that $\Psi(\mathbbm{1}_{\sigma(A)})=A$.
\end{theorem}
\begin{proof}
	
	The case $A=0$ is trivial.
	
	If $A\neq0$, first consider the case where $f$ is a polynomial $p$, $\displaystyle p(x)=\sum_{n=0}^{\infty}\alpha_n x^n$ where $\alpha_n=0$ apart from a finite index set. Define $\displaystyle\Psi(p)(A)=\sum_{n=0}^{\infty}\alpha_n A^n \in \calgebra$.
	
	Now, let $f\in \mathcal{C}\left(\sigma(A)\right)$. If $\|A\|\notin \sigma(A)$, take $\tilde{f}:\mathcal{C}\left(\sigma(A)\cup\{\|A\|\}\right)$ the continuous extension of $f$ satisfying $f(\|A\|)=1$.
	
	From the Weierstrass's Approximation Theorem, for each $i\in \mathbb{N}$ there exists a polynomial $p_i$, defined by $\displaystyle p_n(x)=\sum_{n=0}^{\infty}\alpha^i_n x^n$ where $\alpha^i_n=0$ apart from finite number of $n$'s, such that
	
	$$\left\|p_n-\left(\tilde{f}-3.2^{-n}\right)\right\|<2^{-n}$$
	
	This leads us to conclude that $(p_n)_n$ is a strictly increasing sequence, because, for each $t\in\sigma(A)\cup \{\|A\|\}$,
	$$\begin{aligned}
	\left|p_n(t)-\left(\tilde{f}(t)-3.2^{-n}\right)\right|<2^{-n} \ \forall n \in \mathbb{N} &\Rightarrow
	-2^{-n}<p_n(t)-f(t)+3.2^{-n}<2^{-n} \ \forall n \in \mathbb{N}\\
	& \Rightarrow f(t)-2^{-n+2}<p_n(t)<f(t)-2^{-n+1} \ \forall n\in \mathbb{N} \\
	& \Rightarrow p_n<p_{n+1} \ \forall n \in \mathbb{N}\\
	\end{aligned}$$
	
	Note now that Lemma \ref{lemmax1} gives us $$\|p_i(A)-p_j(A)\|=\left\|(p_i-p_j)(A)\right\|=\left\|(p_i-p_j)(\|A\|\mathbbm{1})\right\|=\left|(p_i-p_j)\left(\|A\|\right)\right|<2^{-n+1} $$
	
	Hence $\left(p_i(A)\right)_i \subset \calgebra$ is a Cauchy's sequence and must converge. The uniqueness of limit in a Hausdorff space allows us to define
	$$\Psi(f)=f(A)=\lim_{i\to \infty}p_i(A) \qquad \forall f\in\mathcal{C}(\sigma(A)) \textrm{ and } p_i\to f \textrm{ uniformly.}$$
	
	All that remains to prove it uniqueness, but it is obvious because if $\Psi_1, \Psi_2$ are such a $\ast$-isomorphisms they must satisfy
	$$\begin{aligned}
	\Psi_1\left(\mathbbm{1}_{\sigma(A)}\right)=A=\Psi_2\left(\mathbbm{1}_{\sigma(A)}\right) \\
	\Psi_1\left(1\right)=\mathbbm{1}=\Psi_2\left(1\right) \\
	\end{aligned}$$
	but then they must coincide in all polynomials which constitute a dense subset, thus $\Psi_1=\Psi_2$.

\end{proof}

\begin{corollary}
	Let $\algebra$ be a $\ast$-algebra. There is at most one norm in $\algebra$ that makes it a $C^\ast$-algebra. In this case, this norm is given by
	$\|A\|=\sqrt{r(A^\ast A)}$.
\end{corollary}

\begin{definition}
	Let $\hilbert$ be a Hilbert space, an operator $A:\hilbert \rightarrow \hilbert \ \in \mathcal{B}(\hilbert)$ is called positive \index{operator! positive} if it is self-adjoint and its spectrum $\sigma(A) \subset \mathbb{R}_+$.
\end{definition}

\begin{proposition}
	\label{positiveeq1}
	$A\in B(\hilbert)$ is a positive operator if, and only if, $A$ is self-adjoint and $\displaystyle \left|\left|1-\frac{A}{\|A\|}\right|\right|\leq 1$.
\end{proposition}
\begin{proof}
	If $A$ is positive it is self-adjoint and $\sigma(A) \subset [0,\|A\|]$, it follows from the Spectral Theorem that 
	$$\sigma\left(\mathbbm{1}-\frac{A}{\|A\|}\right)\subset [0,1] \Rightarrow \displaystyle \left|\left|1-\frac{A}{\|A\|}\right|\right|\leq 1.$$
	
	On the other hand, $\displaystyle \left|\left|1-\frac{A}{\|A\|}\right|\right|\leq 1 \Rightarrow \sigma\left(1-\frac{A}{\|A\|}\right) \subset [-1,1]$, again by the Spectral Theorem it follows that $\sigma(A)\subset [0,2\|A\|]\cap [-\|A\|,\|A\|]$. Moreover, $A$ is self-adjoint and, therefore, positive.
	
\end{proof}

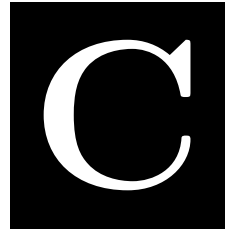
\chapter{Dynamical Systems and Analytic Elements}

\begin{notation}
	Throughout this work we will denote by $\hilbert$ a Hilbert space over $\mathbb{C}$ and by $\calgebra$ a von Neumann algebra over a Hilbert space.
\end{notation}

\begin{definition}[$C^\ast$-Dynamical System]
	A $C^\ast$-Dynamical System $(\calgebra,G,\alpha)$ consists of a $C^\ast$-algebra $\calgebra$, a locally compact group of $\ast$-automorphisms G and a strong-continuous representation $\alpha$ of $G$ in $Aut(\calgebra)$.
\end{definition}

\begin{definition}[$W^\ast$-Dynamical System]
	A $W^\ast$-Dynamical System $(\mathfrak{M},G,\alpha)$ consists of a von Neumann algebra $\mathfrak{M}$, a locally compact group of $\ast$-automorphisms G and a weakly-continuous representation $\alpha$ of $G$ in $Aut(\calgebra)$.
\end{definition}

In particular, we denote by $(\calgebra,\alpha)$ the $C^\ast$-dynamical system with $\alpha$ a one-parameter group, $\mathbb{R} \ni t \mapsto \alpha_t \in Aut(\calgebra)$.

\begin{notation}
	Let $X$ be a Banach Space and $F\subset X^\ast$, where $F$ is such that either $F=X^\ast$ or $F^\ast=X$. We denote by $\sigma(X,F)$ the locally convex topology of $X$ induced by functionals in $F$.
\end{notation}


\begin{definition}[Analytic Elements]
	\label{Anal}
	Let $\alpha$ be a one-parameter $\sigma(X,F)$-continuous group of isometries. An element $A \in X$ is analytic for $\alpha$ if there is a $\gamma>0$ such that
	\begin{enumerate}[(i)]
		{
			\item $f(t)=\alpha_t(A) \ \forall t \in \mathbb{R}$
			\item $z \mapsto \varphi(f(z))$ is analytic in the strip $ \strip{\gamma}=\left\{z \in \mathbb{C} \ \middle | \ |\Im{z}| < \gamma\right\} \ \forall \varphi \in F$
		}
	\end{enumerate}
\end{definition}

Analytic elements will play an important role in proofs because these elements are usually easier to work with and the set of analytic elements is a dense subset.

\begin{proposition}
	\label{AnalDense}
	Let $\alpha$ be a $\sigma(X,F)$-continuous group of isometries and denote by $X_\alpha$ the set of entire elements of $X$ (analytic in the whole $\mathbb{C}$), then $$\overline{X_\alpha}^{\sigma(X,F)}= X .$$
\end{proposition}
\begin{proof}
	Let $A\in X$ and define
	\begin{equation}
	\label{formulaAE}
	A_n(z)=\sqrt{\frac{n}{\pi}}\int_\mathbb{R}{e^{-n(t-z)^2}\alpha_t(A) dt}.
	\end{equation}
	Note that $A_n$ is well defined since $e^{-n(t-z)^2}$ is an integrable function and intuitively $A_n=A_n(0)$ will approach $A$, because the coefficient function approaches a Dirac's delta.
	
	First, note that for $y \in \mathbb{R}$
	\begin{equation}
	\label{1}
	\begin{aligned}
	A_n(y)	&=\sqrt{\frac{n}{\pi}}\int_\mathbb{R}{e^{-n(t-y)^2}\alpha_t(A) dt}\\
	&=\sqrt{\frac{n}{\pi}}\int_\mathbb{R}{e^{-n{t^\prime}^2}\alpha_{t^\prime+y}(A) dt} \\
	&=\alpha_y\left(A_n\right).	\\
	\end{aligned}
	\end{equation}

	Now, in order to show density, suppose $\varphi\in F$ and $\varepsilon>0$ given, then there exists a $\delta>0$ such that $\|\varphi(\alpha_t(A)-A)\|=\|\varphi(\alpha_t(A)-\alpha_0(A))\|\leq \frac{\varepsilon}{2}$ $\forall t\in \left\{t\in \mathbb{R} | |t|<\delta \right\}$, because $\alpha$ is $\sigma(X,F)$-continuous. Now, choose $N \in \mathbb{N}$ such that $\forall n>N$
	$$\sqrt{\frac{n}{\pi}}\int_{\mathbb{R}\setminus{(-\delta,\delta)}}{e^{-nt^2}dt} < \frac{\varepsilon}{4 \|\varphi\| \|A\|}.$$
	
	It follow that, for all $n>N$,
	\begin{equation}
	\label{2}
	\begin{aligned}
	|\varphi(A_n-A)|	& =	\left| \varphi \left( \sqrt{\frac{n}{\pi}}\int_\mathbb{R}{e^{-nt^2}\alpha_t(A)dt}-\sqrt{\frac{n}{\pi}}\int_\mathbb{R}{e^{-nt^2}A \ dt} \right) \right | \\
	& = \left| \sqrt{\frac{n}{\pi}}\int_\mathbb{R\setminus(-\delta,\delta)}{e^{-nt^2}\varphi\left(\alpha_t(A)-A\right)dt}\right| \\
	& \hspace{2.2cm} + \left|\sqrt{\frac{n}{\pi}}\int_\mathbb{(-\delta,\delta)}{e^{-nt^2}\varphi\left(\alpha_t(A)-A\right)dt}\right| \\
	& \leq  \sqrt{\frac{n}{\pi}}\int_{\mathbb{R}\setminus(-\delta,\delta)}{e^{-nt^2}\|\varphi\| \left( \|\alpha_t(A)\|+\|\alpha_0(A)\|\right)dt} \\
	&\hspace{1.95cm}+\sqrt{\frac{n}{\pi}}\int_\mathbb{(-\delta,\delta)}{e^{-nt^2}\|\varphi\left(\alpha_t(A)-A\right)\|dt} \\
	& < \varepsilon.
	\end{aligned}
	\end{equation}
	
	This shows that $A_n \rightarrow A$ in the topology $\sigma(X,F)$. Hence, all that remains is to show that $A_n$ are entire analytic elements, that is, that $A_n(z)$'s are entire analytic. 
	
	Using a similar argument to the one used in equation \eqref{2} is easy to see that $\left(e^{-n(t-z)^2}\varphi\left(\alpha_t(A)\right)\right)_n$ is a Cauchy sequence and consequently it converges pointwise. Using the inequality $\left|\varphi\left(\tau_t(A)\right)\right|\leq \|\varphi\| \|A\|$, we conclude that $e^{-n(t-z)^2}\varphi\left(\tau_t(A)\right)$ is dominated by $e^{-n(t-z)^2} \|\varphi\| \|A\|$, which is measurable. It follows from the Lebesgue dominated convergence theorem that $\varphi\left(\tau_y(A)\right)$ is analytic.
	
\end{proof} 


\end{appendices}

\begin{spacing}{0.9}





\printbibliography[heading=bibintoc, title={References}]

\end{spacing}



\end{document}